\documentclass[final]{siamart220329}
\usepackage{hyperref}
\usepackage{epsfig}
\usepackage{subfigure}
\usepackage{amsfonts}
\usepackage{algorithm}
\usepackage{algorithmic}
\usepackage{graphicx}
\usepackage{subfigure}
\usepackage{array}
\usepackage{enumerate}
\usepackage{enumitem}
\usepackage{rotating}
\usepackage{multirow}
\usepackage[T1]{fontenc}
\AtBeginDocument{
    \hypersetup{colorlinks,urlcolor=black}
}

\usepackage{pgfplots,tikz}

\newdimen\prevdp
\def\leftlabel#1{\noalign{\prevdp=\prevdepth
		\kern-\prevdp\nointerlineskip\vbox to0pt{\vss\hbox{#1}}\kern\prevdp}}

\newcommand {\mre}	{\mathrm{Re}}
\newcommand {\mri}	{\mathrm{Im}}

\newtheorem{example}[theorem]{Example}

\newtheorem{remark}[theorem]{Remark}
\newtheorem{assumption}[theorem]{Assumption}

\newcommand{\sn}{\mathsf{n}}
\newcommand{\sm}{\mathsf{m}}
\newcommand{\scc}{\mathsf{c}}
\newcommand{\sk}{\mathsf{k}}

\newcommand{\sd}{\mathsf{d}}

\newcommand{\stan}{\mathsf{st}}
\newcommand{\EIGOPT}{\textsf{EigOpt}}
\newcommand{\GRANSO}{\textsf{GRANSO}}

\usepackage{todonotes}

\headers{V. Mehrmann and E. Mengi}{Robust Stability of a Quadratic Matrix Polynomial}

\title{Minimization of the Pseudospectral Abscissa of a Quadratic Matrix Polynomial}
\author{Volker Mehrmann\footnotemark[2] \and Emre Mengi\footnotemark[3]}
\date{\today}

\begin{document}
\maketitle

\begin{abstract}
\noindent
For a quadratic matrix polynomial dependent on parameters and a given tolerance $\epsilon > 0$, the minimization of the $\epsilon$-pseudospectral abscissa over the set of
permissible parameter values is discussed, with applications in damping optimization and brake squeal reductions in mind. 
An approach is introduced that is based on nonsmooth and global optimization
(or smooth optimization techniques such as BFGS if there are many parameters)
equipped with a globally convergent criss-cross algorithm to compute the $\epsilon$-pseudospectral
abscissa objective when the matrix polynomial is of small size. 
For the setting when the matrix polynomial is large,  a subspace framework
is introduced, and it is argued formally that
it solves the minimization problem globally. The subspace framework 
restricts the parameter-dependent matrix polynomial to small subspaces, and thus solves
the minimization problem for such restricted small matrix polynomials. It then expands the
subspaces using the minimizers for the restricted polynomials. The proposed approach
makes the global minimization of the $\epsilon$-pseudospectral abscissa possible for a quadratic matrix polynomial dependent on a few parameters and for sizes up to at least a few hundreds. This is illustrated
on several examples originating from damping optimization.
\end{abstract}



\footnotetext[2]{
Institut f\"ur Mathematik MA 4-5, TU Berlin, Str. des 17. Juni 136,
D-10623 Berlin, FRG.
\texttt{mehrmann@math.tu-berlin.de}. }
\footnotetext[3]{Ko\c{c} University, Department of Mathematics, Rumeli Feneri Yolu 34450, Sar{\i}yer, Istanbul, Turkey.
\texttt{emengi@ku.edu.tr}.}

\medskip 

\section{Introduction}

We consider a quadratic matrix polynomial that depends on parameters, and,
in particular, we are interested in minimizing the $\epsilon$-pseudospectral
abscissa of such a matrix polynomial over the parameters. Our work is motivated by
applications such as damping optimization \cite{NakVT23,Tom23} 
and the brake squealing problem \cite{GraMQ2015}.
Having a negative $\epsilon$-pseudospectral abscissa after minimization
guarantees robust stability, i.e. the asymptotic stability of not only the associated second-order 
linear time-invariant system, but also all nearby systems.

The \emph{parameter-dependent matrix polynomial} under consideration is of the form
\begin{equation}\label{eq:par_dep_mat_poly}
	P(\lambda; \nu)
		\;	=	\;	\lambda^2 M(\nu) \: + \:  \lambda C(\nu) \: + \:	K(\nu)
\end{equation}
where $\nu \in \Omega \subseteq {\mathbb R}^{\sd}$ represents a vector of parameters,
$\Omega$ is an open subset of ${\mathbb R}^{\sd}$, and the coefficients have the special form
\begin{equation}\label{eq:par_dependent_mats}
	\begin{split}
	&
	M(\nu)	
		\;	=	\;\:	\psi_{1}(\nu) M_{1} +	\dots		+	\psi_{\sm}(\nu)	M_{\sm}	\,	,	\\[.3em]
	&
	C(\nu)	
	\;\:	=	\;\:	\zeta_{1}(\nu) C_{1} +	\dots		+	\zeta_{\scc}(\nu)	 C_{\scc}	\,	,	\\[.3em]
	&
	K(\nu)	
		\;\,	=	\;\:	\kappa_{1}(\nu) K_{1} +	\dots		+	\kappa_{\sk}(\nu)	 K_{\sk}	\,	,
	\end{split}
\end{equation}
for given (parameter-independent) matrices $M_{i}, C_{j}, K_{\ell} \in {\mathbb C}^{\sn\times \sn}$, and real-analytic functions
$\psi_{i}, \zeta_j , \kappa_\ell : \Omega \rightarrow {\mathbb R}$ for $i = 1, \dots , \sm$,
$j = 1, \dots , \scc$, $\: \ell = 1, \dots , \sk$. 
In a context where the dependence
of $P(\lambda; \nu)$ on $\nu$ is not important, we may leave out this explicit dependence on $\nu$, and simply write
\begin{equation}\label{eq:mat_poly}
	P(\lambda) \; = \;  \lambda^2 M + \lambda C + K
\end{equation}	
for matrices $M, C, K \in {\mathbb C}^{\sn\times \sn}$. 

There are slightly different (but closely related) definitions of the \emph{$\epsilon$-pseudospectrum} of the matrix polynomial $P(\lambda)$ 
in (\ref{eq:mat_poly}) in the literature, \cite{Men06,TisH01,TisM01}.
In this work, for computational convenience, we use the definition from
 \cite{Men06}:
\begin{equation}\label{eq:defn_psa}
	\begin{split}
	\Lambda_{\epsilon}(P)
			\;	:= &		\;
	\{	\lambda \in {\mathbb C}	\;	|	\;
			\text{det}(P(\lambda) + \Delta P(\lambda)) = 0 \\[.2em]
& \;\;\;\;\;
\text{ for some } \Delta P(\lambda) = w_m \lambda^2 \Delta M + w_c \lambda \Delta C	+ w_k \Delta K  \\[.2em]
& \;\;\;\;\;
\text{ with } \, \Delta M , \Delta C , \Delta K \in {\mathbb C}^{\sn\times \sn}  \;
			\text{ s.t. } 
				\| \left[
					\begin{array}{ccc}
					\Delta M	&	\Delta C	&	\Delta K
					\end{array}
				\right] \|_2 \leq \epsilon 	
	\}		\;	,
	\end{split}
\end{equation}
with $\| \cdot \|_2$ representing the matrix 2-norm, and
nonnegative real scalars $w_m, w_c, w_k$ that can be regarded as weights.
It follows from the arguments in the proof of \cite[Lemma 8]{GenSVD02}
that $\Lambda_{\epsilon}(P)$ can be characterized as
\begin{equation}\label{eq:char}
	\Lambda_{\epsilon}(P)
		\;	=	\;
	\{
		\lambda \in {\mathbb C}	\;	|	\;
			\sigma_{\min}(P(\lambda))		\leq	\epsilon p_w(| \lambda |)	
	\} 
\end{equation}
where $\sigma_{\min}(A)$ denotes the smallest singular value of the matrix $A$, and	
\begin{equation}\label{eq:scaling_function}
	p_w(z)		\; := \;	 \sqrt{w_m^2 z^4 + w_c^2 z^2 + w_k^2} ,
\end{equation}
as well as $w := (w_m , w_c , w_k)$. 

Extended definitions of the $\epsilon$-pseudospectrum for more general  polynomial eigenvalue problems,
as well as general nonlinear eigenvalue problems are also presented in the literature \cite{Men06,Mic06,TisH01,TisM01}. 

The \emph{$\epsilon$-pseudospectral abscissa} of $P$ is defined as
\begin{equation}\label{eq:defn_ps_abs}
	\begin{split}
	\alpha_{\epsilon}(P)
		\;	&	:=	\;
	\sup \{ \mre(z)	\;	|	\;		z	\in	\Lambda_\epsilon(P)		\}	\\[.3em]
		\;	&	\phantom{:}=	\;
	\sup \{ \mre(z)	\;	|	\;		z	\in	{\mathbb C}	\text{ s.t. }
						\sigma_{\min}(P(z))		\leq	\epsilon p_w(| z |)
			\}.
	\end{split}
\end{equation}
The main problem that we deal with is the minimization of the $\epsilon$-pseudospectral
abscissa of $P(\lambda; \nu)$ in (\ref{eq:par_dep_mat_poly}) over the set of permissible
parameter values, i.e.
\[
	\min\{ \, \alpha_{\epsilon}(P( \, \cdot \, ; \nu))	\;	|	\;	\nu	\in \underline{\Omega} 	\, \}
\]
with $\underline{\Omega} \subseteq \Omega$ denoting a compact set of permissible
parameter values. We propose approaches for both small-scale problems when $\sn$ is small, 
and moderate- to large-scale problems when $\sn$ is not small. Throughout this work,
we keep the following boundedness assumption.
\begin{assumption}\label{ass:bounded_ps}
$\Lambda_{\epsilon}(P( \, \cdot \, ; \nu))$ is bounded for all $\nu \in \underline{\Omega}$.
\end{assumption}
Under this boundedness assumption, a rightmost point in $\Lambda_{\epsilon}(P( \, \cdot \, ; \nu))$
must exist, i.e. the supremum in the definition of $\alpha_{\epsilon}(P( \, \cdot \, ; \nu))$
(see (\ref{eq:defn_ps_abs}) above) must be attained, so it can be replaced by the maximum 
for all $\nu \in \underline{\Omega}$.

The ideas and approaches presented in this paper can be extended in a straightforward way to a parameter-dependent
nonlinear eigenvalue problem when $P(\lambda; \nu)$ in (\ref{eq:par_dep_mat_poly})
is replaced by a more general matrix-valued function
\begin{equation}\label{eq:NLEVP}
	T(\lambda; \nu)
		\;	=	\;	g_0(\lambda) A_0(\nu)  \: + \:  g_1(\lambda) A_1(\nu)	\: + \: 	\dots	 \: + \:
						g_\ell(\lambda) A_\ell(\nu)
\end{equation}
for given meromorphic functions $g_j : {\mathcal D} \rightarrow {\mathbb C}$ for $j = 0, 1, \dots, \ell$ 
with a domain ${\mathcal D}$ that is a dense subset of ${\mathbb C}$.
The polynomial eigenvalue problem is a special case of (\ref{eq:NLEVP})
with $g_j(\lambda) = \lambda^j$ for $j = 0, 1, \dots , \ell$.

To have a concise presentation, we restrict ourselves to quadratic matrix polynomials 
with $g_j(\lambda) = \lambda^j$ for $j = 0, 1, 2$.
The outcome of our work is an efficient approach to optimize the robust stability 
of a parameter-dependent quadratic eigenvalue problem, viewed as an important
challenging problem for instance by the damping optimization community \cite{Nak02,NakTT13,Tom23}.
The new approach is designed systematically, especially well-suited for large problems, 
and has theoretical global convergence guarantee.

The current work can be regarded as a generalization of \cite{AliM22} that concerns
the minimization of the $\epsilon$-pseudospectral abscissa of a matrix dependent
on parameters. We extend the analytical properties of the $\epsilon$-pseudospectral
abscissa functions from the matrix setting to the quadratic matrix polynomial setting.
The algorithms are built on top of these analytical properties, and implementing them
requires considerable effort and involves many subtleties.


Our method comes with publicly available MATLAB implementations \cite{MehM24} 
of the proposed algorithms. We illustrate the effectiveness of the proposed algorithms 
in practice by applying them 
to real damping optimization problems.

\medskip

\noindent
\textbf{Outline.} In Section~\ref{sec:small_probs} we propose algorithms for small problems
involving coefficients of small dimension $\sn$.  We first describe a
globally convergent criss-cross algorithm to compute the $\epsilon$-pseudospectral abscissa 
and the rightmost point in the $\epsilon$-pseudospectrum. 
Then, we discuss how $\alpha_{\epsilon}(P(\cdot ; \nu))$ can be minimized in the small-scale
setting, where we present conditions ensuring the real analyticity of $\alpha_{\epsilon}(P(\cdot ; \nu))$,
and derive expressions for the derivatives of $\alpha_{\epsilon}(P(\cdot ; \nu))$.

Section~\ref{sec:large_compute} is devoted to large-scale computation of the
$\epsilon$-pseudospectral abscissa and the corresponding rightmost point,
in particular to the development of a subspace framework for this purpose when the dimension of the matrix polynomial is not small. 

Section~\ref{sec:large_minimize} focuses on
the minimization of $\alpha_{\epsilon}(P(\cdot ; \nu))$ in the case that the parameter-dependent
matrix polynomial at hand is not of small dimension. A second subspace framework for minimizing
$\alpha_{\epsilon}(P(\cdot ; \nu))$ is introduced, and a formal argument is given
in support of its convergence to a global minimizer of $\alpha_{\epsilon}(P(\cdot ; \nu))$.

Section~\ref{sec:apply_damping_opt} concerns the applications of the proposed approaches
for $\epsilon$-pseudospectral abscissa minimization to several examples 
originating from damping optimization. An alternative to directly minimizing
$\alpha_{\epsilon}(P(\cdot ; \nu))$ over $\nu$ is minimizing the $\epsilon$-pseudospectral
abscissa of a linearization of $P(\cdot ; \nu)$ over $\nu$. We provide a comparison of
these two $\epsilon$-pseudospectral abscissa minimization problems 
in Section \ref{sec:psa_linearizations}, and an example for which the minimizers of
these two minimization problems differ considerably.

We summarize our results, and potential directions for future research in Section~\ref{sec:conclusion}.

\section{Small-scale problems}\label{sec:small_probs}
In this section, we deal with a matrix polynomial that has small size,
and describe approaches for the computation and minimization
of its $\epsilon$-pseudospectral abscissa.
\subsection{Computation of the $\epsilon$-pseudospectral abscissa for small-scale problems}
Inspired by the ideas in \cite{BurLO05} for the matrix case, in \cite{Men06} a criss-cross algorithm is 
presented to compute the $\epsilon$-pseudospectral abscissa $\alpha_{\epsilon}(P)$ for a small-scale
matrix polynomial $P(\lambda)$. We briefly recall this algorithm starting with its
two main ingredients, namely a vertical search and a horizontal search.

The vertical search for a given $x \in {\mathbb R}$, stated below as a theorem, is helpful
to locate the intersection points of the boundary of $\Lambda_{\epsilon}(P)$ in the complex plane
with the vertical line $\mre(z) = x$.
\begin{theorem}[Vertical Search \cite{Men06}]\label{thm:vert_src}
For every $x \in {\mathbb R}$, the following statements are equivalent:
\begin{enumerate}
\item[\bf (i)]
One of the singular values of $\, \{ 1/  p_w( | x + {\rm i} y |) \} P(x + {\rm i} y)$ for $y \in {\mathbb R}$
equals $\epsilon$.
\item[\bf (ii)]
The matrix polynomial ${\mathcal M}(\lambda; x) = \sum_{j=0}^{4} \lambda^j {\mathcal M}_j(x)$
with
\begin{equation*}
	{\mathcal M}_j(x)
		\;	:=		\;
	\left\{
		\begin{array}{cl}
		\left[
			\begin{array}{cc}
				-\epsilon p^2_w(x)	I	&	\; [P(x)]^\ast		\\[.25em]
				P(x)					&	-\epsilon I
			\end{array}
		\right]	&	\quad			\text{if } 	j	=	0	\;	,		\\[1.3em]
		\left[
			\begin{array}{cc}
				0				&	\; -[ P'(x) ]^\ast	\\[.25em]
				P'(x)				&	\; 0
			\end{array}
		\right]	&	\quad			\text{if } 	j	=	1	\;	,	\\[1.6em]
		\left[
			\begin{array}{cc}
			\epsilon (2 w_m^2 x^2		+	w_c^2)  I		&	\; M^\ast	\\[.25em]
			M				&	0
			\end{array}
		\right]	&	\quad			\text{if } 	j	=	2	\;	,		\\[1.9em]
		0		&	\quad			\text{if } 	j	=	3	\;	,		\\[1.3em]
		\left[
			\begin{array}{cc}
				-\epsilon w_m^2 I				&	\; 0		\\
					0			&	\; 0
			\end{array}
		\right]	&	\quad			\text{if } 	j	=	4	\;
		\end{array}
	\right.
\end{equation*}
has ${\rm i} y$ as an eigenvalue.
\end{enumerate} 
\end{theorem}
According to Theorem~\ref{thm:vert_src}, the set of intersection points of the boundary of $\Lambda_{\epsilon}(P)$
with the vertical line $\mre(z) = x$ can be obtained by extracting the purely imaginary
eigenvalues ${\rm i} y_1, \dots {\rm i} y_{\eta(x)}$ of ${\mathcal M}(\lambda; x)$,
then checking whether
\[
	 \sigma_{\min}(P(x + {\rm i} y_j))
	 	\;	=	\;	\epsilon \left\{ p_w( | x + {\rm i} y_j | ) \right\}
\]
for $j = 1, \dots , \eta(x)$.

We next state, as a theorem, the horizontal search, which concerns finding the
rightmost intersection point of the $\epsilon$-pseudospectrum with a horizontal line
in the complex plane.
\begin{theorem}[Horizontal Search \cite{Men06}]\label{thm:horiz_src}
For a given $y \in {\mathbb R}$, if the set 
\[
	{\mathcal H}(y)
		\;	:=			\;
	 \Lambda_{\epsilon}(P)  \cap  \left\{ \mri(z) = y \right\}
\]
is not empty, then the matrix polynomial
${\mathcal N}(\lambda; y) = \sum_{j=0}^{4} \lambda^j {\mathcal N}_j(y)$
with
\begin{equation*}
	{\mathcal N}_j(y)
		\;	:=		\;
	\left\{
		\begin{array}{cl}
		\left[
			\begin{array}{cc}
				-\epsilon p^2_w(y)	I	&	\;\; [P({\rm i} y)]^\ast		\\[.25em]
				P({\rm i} y)					&	-\epsilon I
			\end{array}
		\right]	&	\quad			\text{if } 	j	=	0	\;	,		\\[1.3em]
		\left[
			\begin{array}{cc}
				0				&	\; -{\rm i} [ P'({\rm i} y) ]^\ast	\\[.25em]
				-{\rm i} P'({\rm i} y)				&	\;\; 0
			\end{array}
		\right]	&	\quad			\text{if } 	j	=	1	\;	,		\\[1.6em]
		\left[
			\begin{array}{cc}
	\epsilon (2 w_m^2 y^2		+	w_c^2)  I				&	\; -M^\ast	\\[.25em]
	-M				&	\;\; 0
			\end{array}
		\right]	&	\quad			\text{if } 	j	=	2	\;	,		\\[1.9em]
		0		&	\quad			\text{if } 	j	=	3	\;	,		\\[1.3em]
		\left[
			\begin{array}{cc}
				-\epsilon w_m^2 I				&	\; 0		\\
					0			&	\; 0
			\end{array}
		\right]	&	\quad			\text{if } 	j	=	4
		\end{array}
	\right.	
\end{equation*}
has a purely imaginary eigenvalue. Moreover, if ${\mathcal H}(y)$ is not empty, then
\begin{equation}\label{eq:right_psa}
	\max \, \mre\{ {\mathcal H}(y) \}
		\;	=	\;
	\max\{ x	\;	|	\;	x \in {\mathbb R} \, \text{ s.t. } \,
							{\rm i} x  \in \Lambda({\mathcal N}(\lambda; y))	\}	
\end{equation}
where $\Lambda({\mathcal N}(\lambda; y))$ denotes the set of finite eigenvalues of 
the matrix polynomial ${\mathcal N}(\lambda; y)$.
\end{theorem}
According to Theorem~\ref{thm:horiz_src}, the rightmost intersection point of $\Lambda_{\epsilon}(P)$
with the horizontal line $\mri(z) = y$ in the complex plane can be determined from the 
set of purely imaginary eigenvalues of ${\mathcal N}(\lambda; y)$ by employing (\ref{eq:right_psa}).

The criss-cross algorithm to compute $\alpha_{\epsilon}(P)$ at every iteration
for a given estimate $\underline{x}$ first finds the intervals on the line $\mre(z) = \underline{x}$
that lie in $\Lambda_{\epsilon}(P)$ by means of a vertical search. Then a horizontal
search is performed at the midpoint of each such interval, and the estimate $\underline{x}$
is updated to the real part of the rightmost point in $\Lambda_{\epsilon}(P)$
obtained from these horizontal searches. A formal description of the approach 
is given in Algorithm \ref{alg:criss-cross}. 
\begin{algorithm}
\begin{algorithmic}[1]
	\REQUIRE{A quadratic matrix polynomial $P(\lambda)$ of the form (\ref{eq:mat_poly}),
				and $\epsilon > 0$.}
	\ENSURE{Estimates $\underline{x}$ for $\alpha_{\epsilon}(P)$, and $\underline{z} \in {\mathbb C}$
				for a rightmost point in $\Lambda_{\epsilon}(P)$.}	
	\vskip 1.2ex
	\STATE{$\underline{z} \gets$ a rightmost eigenvalue of $P(\lambda)$.}\label{alg:init_points2}
	\vskip .4ex
	\STATE{$\underline{x}^{(0)} \gets$ $\text{Re}(\underline{z})$.}\label{alg:init_points}
	\vskip 1.2ex
	\FOR{$k=1,2,\dots$}
	\vskip 1.2ex
	\STATE{\textbf{Vertical Scan.} \\
		Determine the intervals $[y^L_j , y^R_j]$ for $j = 1, \dots , l$
	such that
	\begin{equation}\label{eq:vs}
		\bigcup_{j=1}^{l} [y^L_j , y^R_j]
			\;	=	\;	
		 \mri\left\{ \Lambda_{\epsilon}(P)  \cap \left\{ \mre(z) = \underline{x}^{(k-1)} \right\} \right\}
	\end{equation}
	by means of a vertical search at $\underline{x}^{(k-1)}$.
	}\label{line:vscan} 
	\vskip 1.2ex
	\STATE{\textbf{Horizontal Refinement.} \\
	Letting $y^{\mathrm{mid}}_j := (y^L_j + y^R_j)/2$, set
		\begin{equation}\label{eq:hs}
			\underline{x}^{(k)} \gets 
			\max_{j = 1, \dots , l} 
			\max \left\{ 
			\mre \left\{ \Lambda_{\epsilon}(P)  \cap \left\{ \mri(z) = y^{\mathrm{mid}}_j \right\} \right\}
			\right\}
		\end{equation}
	by means of horizontal searches at $y^{\mathrm{mid}}_j$
	for $j = 1, \dots , l$, and set \\[.2em]
	$	\,
		\underline{z} \gets \underline{x}^{(k)} + {\rm i} \, y^{\mathrm{mid}}_{\underline{j}}	\,
	$
	with $\underline{j}$ denoting the maximizing $j$ in (\ref{eq:hs}).}\label{line:href}
	\vskip 1.9ex
	\STATE{\textbf{Return} $\underline{x} \gets \text{Re}(\underline{z}) \:$ and
	$\, \underline{z} \;$  if convergence occurred.} \label{alg:terminate}
	\vskip 1.5ex
	\ENDFOR
\end{algorithmic}
\caption{The criss-cross algorithm to compute $\alpha_{\epsilon}(P)$}
\label{alg:criss-cross}
\end{algorithm}

An example of the progress of Algorithm \ref{alg:criss-cross} is presented
for a $4\times 4$ quadratic matrix polynomial in Figure \ref{fig:criss-cross_illustrate}. In this example, 
the algorithm first finds the intersection points of the line $\mre (z) = \alpha(P)$ with the 
boundary of $\Lambda_{\epsilon}(P)$ by means of a vertical search,
where $\alpha(P)$ denotes the spectral abscissa of $P$, i.e. the real part of
the rightmost eigenvalue of $P$ (see the left-hand plot). 
Using these intersection points, it determines the three intervals on the line $\mre (z) = \alpha(P)$ 
that lie inside $\Lambda_{\epsilon}(P)$, and performs a horizontal search at the midpoint 
of each one of these three intervals (see the middle plot). 
Then the algorithm performs another vertical search to find the intersection points of the 
boundary of $\Lambda_{\epsilon}(P)$ with the line $\mre (z) = \underline{x}$, where $\underline{x}$ is 
the real part of the rightmost point returned by the three horizontal searches 
(see the right-hand plot).

\begin{figure}
 \centering
	\begin{tabular}{ccc}
		\hskip -6ex
			\includegraphics[width = .37\textwidth]{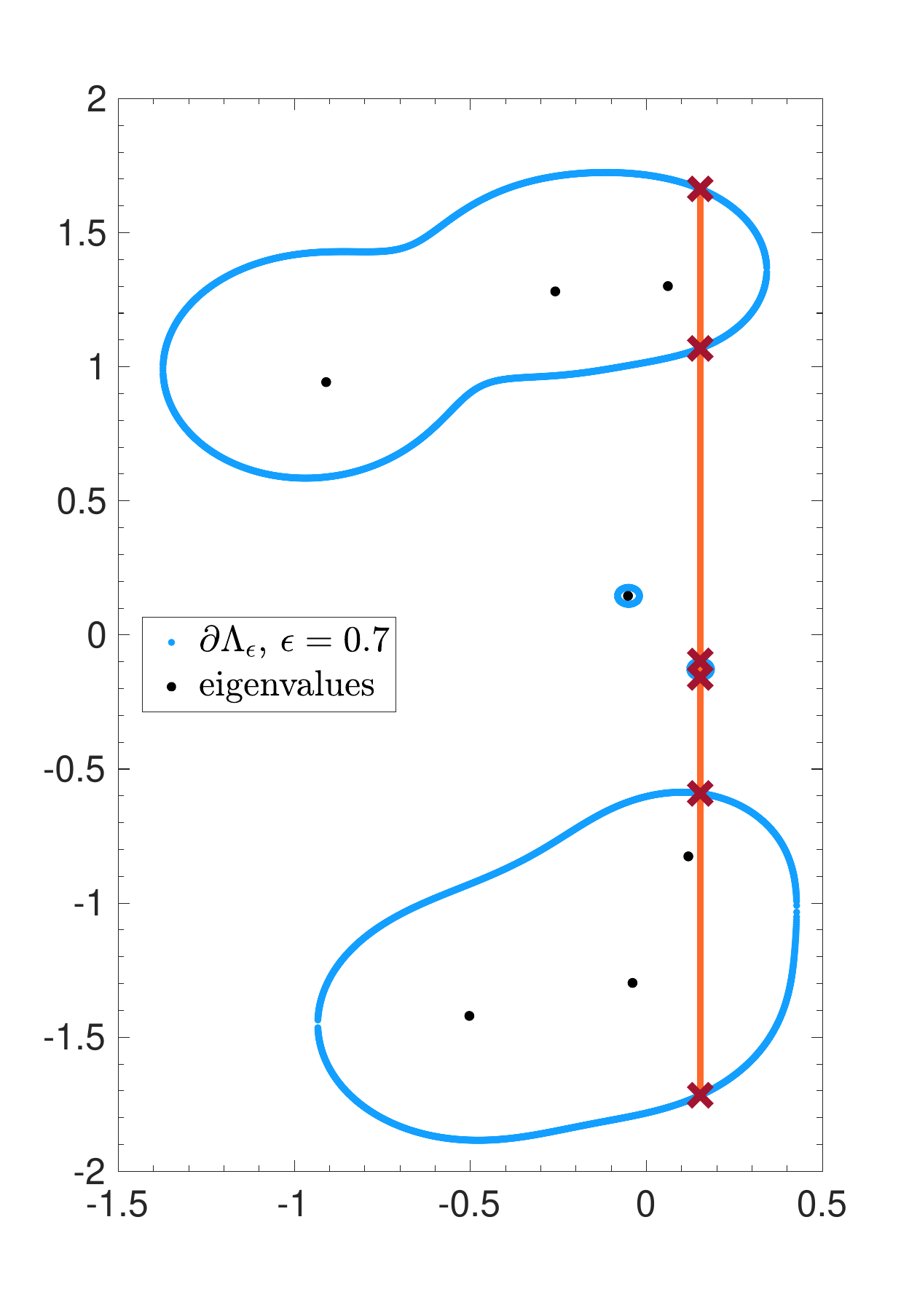} & 
			\hskip -2ex
			\includegraphics[width = .37\textwidth]{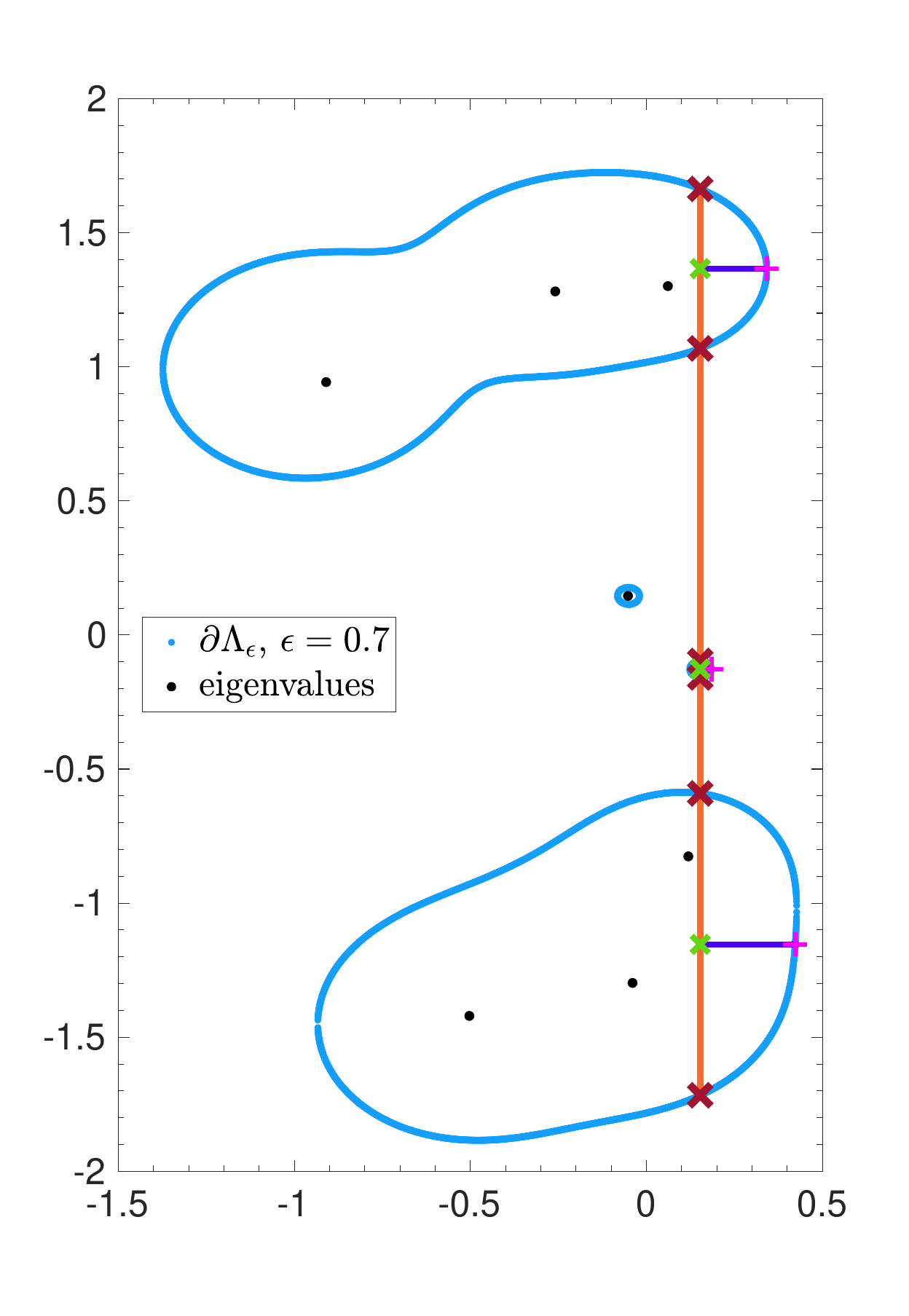} &
			\hskip -2ex
			\includegraphics[width = .37\textwidth]{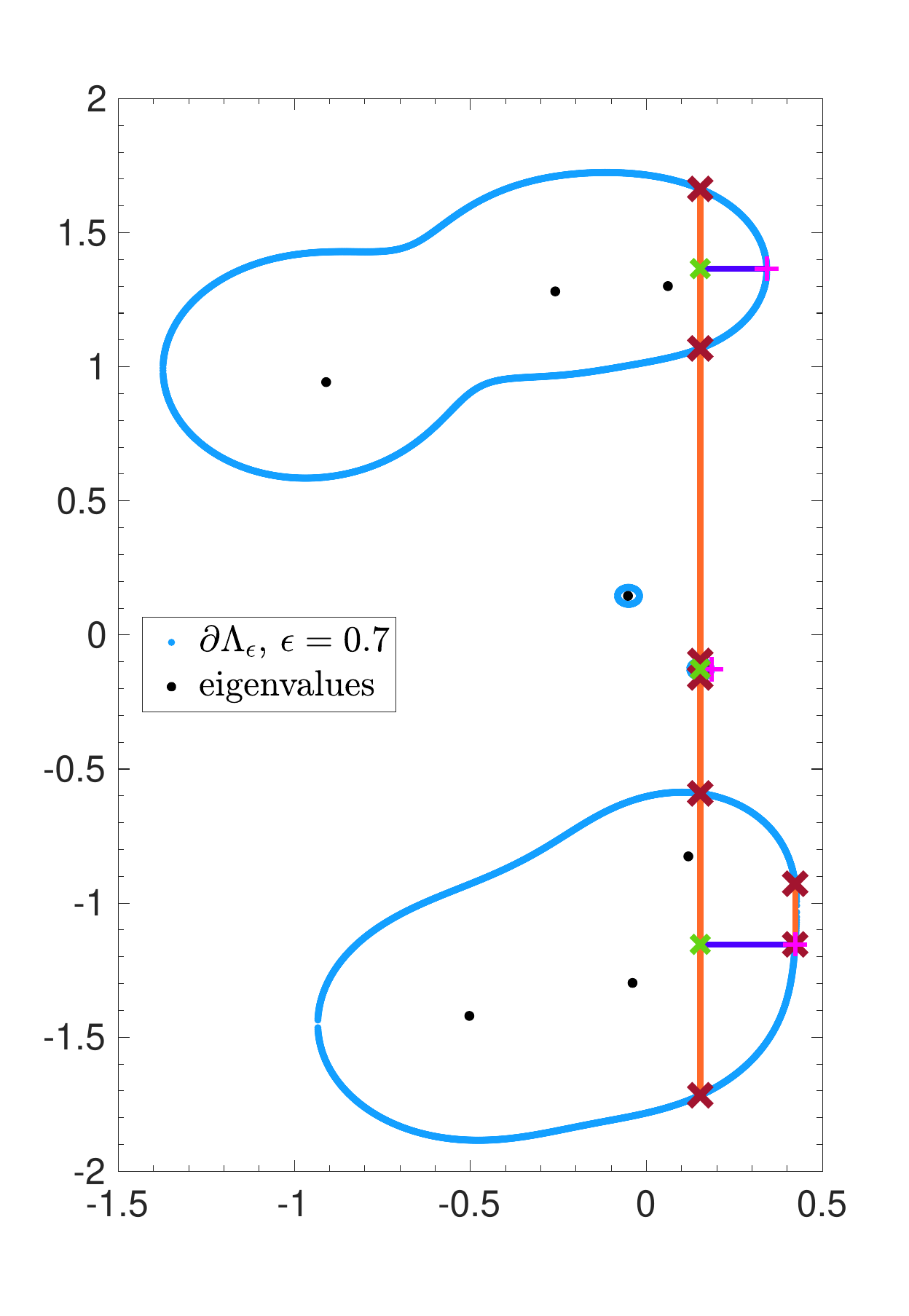}
	\end{tabular}
		\caption{  Progress of Algorithm \ref{alg:criss-cross} 
		for a $4\times 4$ quadratic matrix polynomial. The blue curves
		correspond to the boundary $\partial \Lambda_{\epsilon}$ of the
		$\epsilon$-pseudospectrum of the quadratic matrix polynomial 
		for $\epsilon = 0.7$.
		}
		\label{fig:criss-cross_illustrate}
\end{figure}

\textbf{Convergence.}
It has been shown in \cite{BurLO05} that the estimate $x^{(k)}$ 
generated by Algorithm~\ref{alg:criss-cross}
in the matrix case converges to the $\epsilon$-pseudospectral 
abscissa as $k \rightarrow \infty$, and the rate of this convergence is at least quadratic. 
These results extend to Algorithm \ref{alg:criss-cross}, that is to the
criss-cross algorithm for the more general quadratic matrix 
polynomial setting. The proofs of the convergence results 
stated below for Algorithm \ref{alg:criss-cross} are similar to their 
counterparts in the matrix setting,
in particular to \cite[Theorem 3.2]{BurLO05} and \cite[Theorem 5.2]{BurLO05}, 
but by replacing the singular value function $h(x,y)$ in there with
\begin{equation}\label{eq;defn_g}
	g(x,y) \; := \; \sigma_{\min}(P(x + {\rm i} y)) - \epsilon p_w(\sqrt{x^2 + y^2}) .
\end{equation}
These results are stated assuming there is no termination, i.e. 
for Algorithm \ref{alg:criss-cross} without line \ref{alg:terminate}.
\noindent
\begin{theorem}[Global Convergence]
The sequence $\{ x^{(k)} \}$ generated by Algorithm \ref{alg:criss-cross} is monotonically
non-decreasing and satisfies $\: \lim_{k\rightarrow \infty} x^{(k)} = \alpha_{\epsilon}(P)$.
\end{theorem}
\begin{theorem}[At Least Quadratic Convergence]
Suppose that at every globally rightmost point 
$z_\ast = x_\ast + {\rm i} y_\ast$ with 
$x_\ast, y_\ast \in {\mathbb R}$ of $\Lambda_{\epsilon}(P)$
the singular value $\sigma_{\min}(P(z_\ast))$ is simple and 
$\nabla g(x_\ast, y_\ast) \neq 0$. Then the sequence $\{ x^{(k)} \}$ 
generated by Algorithm \ref{alg:criss-cross} satisfies
\[
	| x^{(k+1)} - \alpha_{\epsilon}(P) |	\;	\leq		\;   
		{\mathcal C} | x^{(k)} - \alpha_{\epsilon}(P) |^2
\]
for some constant ${\mathcal C} \geq 0$.
\end{theorem}
Table \ref{table:criss-cross_iterates} lists the iterates of the algorithm on two exemplary 
quadratic matrix polynomials. For both polynomials the iterates converge at a quadratic rate
as expected.

\begin{table}
\hskip 15ex
 \begin{tabular}{|c||c|}
        \hline
        $k$ & $x^{(k)}$ \\
        \hline
        0 & 0.152872406397 \\ 
        1 & 0.\underline{42}2554211904 \\ 
        2 & 0.\underline{426}672961073 \\ 
        3 & 0.\underline{426804}683720 \\ 
        4 & 0.\underline{426804822541} \\ 
        \hline
 \end{tabular}
 \hskip 2.5ex
 \begin{tabular}{|c||c|}
        \hline
        $k$ & $x^{(k)}$ \\
        \hline
        0 & 0.731311985577 \\ 
        1 & 0.\underline{8552}40708191 \\ 
        2 & 0.\underline{8552562255}58 \\ 
        3 & 0.\underline{855256225560} \\ 
        \phantom{4} & 	\\
        \hline
    \end{tabular}
    \caption{ Lists of iterates of Algorithm \ref{alg:criss-cross} on two
    quadratic matrix polynomials. Left: Polynomial of size $4\times 4$, $\epsilon = 0.7$
    and $(w_m, w_c, w_k) = (1,1,1)$. Right: Polynomial of size $100\times 100$,
    $\epsilon = 0.3$ and $(w_m, w_c, w_k) = (0.4,0.3,1)$.}
    \label{table:criss-cross_iterates}
    \end{table}

\textbf{Termination.}
We terminate the algorithm in line \ref{alg:terminate} if the estimate $x^{(k)}$ 
after update in line \ref{line:href} is not greater than the previous estimate $x^{(k-1)}$. 
Another possibility for termination is that the vertical search in line \ref{line:vscan} 
cannot find any intersection points of $\Lambda_{\epsilon}(P)$
with $\mre(z) = x^{(k-1)}$, i.e. the matrix polynomial ${\mathcal M}(\lambda ; x^{(k-1)})$
does not have any imaginary eigenvalue. In this second possibility, 
we simply skip the horizontal refinement
in line \ref{line:href} and terminate in line \ref{alg:terminate}.

\textbf{Computational cost.}
The computationally most expensive parts
of Algorithm \ref{alg:criss-cross} are the vertical search in line \ref{line:vscan} 
and $l$ horizontal searches in line \ref{line:href} that need to be performed
at every iteration. Each vertical and horizontal search typically requires 
the computation of all eigenvalues of a quartic matrix polynomial of size $2\sn \times 2\sn$, 
that can be determined at an ${\mathcal O}(\sn^3)$ computational cost in practice.
Hence, every iteration of Algorithm \ref{alg:criss-cross} has a cubic cost 
${\mathcal O}(\sn^3)$ in practice, yet the constant hidden in the 
${\mathcal O}$-notation is large, since the eigenvalue problems are of size 
$2\sn \times 2\sn$ and involve quartic matrix polynomials. This makes the 
algorithm suitable only for small matrix polynomials.

\subsection{Minimization of the $\epsilon$-pseudospectral abscissa for small-scale problems}\label{sec:min_full}
The pseudospectral abscissa function $\alpha_{\epsilon}(\nu) := \alpha_{\epsilon}(P(\, \cdot \,; \nu))$ 
is usually not convex in the parameters $\nu$. Moreover, it is typically not smooth whenever the
rightmost point in $\Lambda_{\epsilon}(P(\, \cdot \,; \nu))$ is not unique, or the singular value
$\sigma_{\min}(P(\underline{z} \,; \nu))$ at the rightmost point $\underline{z}$ of 
$\Lambda_{\epsilon}(P(\, \cdot \,; \nu))$ is not simple.
Smooth optimization techniques, especially BFGS, have been used successfully on such nonsmooth
optimization problems; see in particular \cite{CurMO17} and 
the MATLAB package {\GRANSO}, which is based on BFGS. 
Such smooth optimization techniques operate locally, and can at best be expected to converge to a local 
minimizer of $\alpha_{\epsilon}(\nu)$, that is likely not a global minimizer. On the other end,
global optimization techniques may guarantee convergence to a global minimizer of 
$\alpha_{\epsilon}( \nu)$, but they become computationally infeasible
if there are more than a few parameters. Such a global optimization technique based on global
piecewise quadratic approximations \cite{BreC1993} is adopted
for eigenvalue functions in \cite{MenYK14}. Its implementation {\EIGOPT} is publicly available. 
At every iteration, this approach  minimizes the maximum of several quadratic functions 
bounding the objective from below and interpolating the objective at several points, 
then adds a new quadratic function that
interpolates the objective at the computed minimizer to existing quadratic functions,
and again minimizes the maximum of quadratic functions.
We emphasize that both {\GRANSO} and {\EIGOPT} work well for nonsmooth problems,
as long as the objective is continuous, piecewise continuously differentiable and differentiable
almost everywhere, which is the case for $\alpha_{\epsilon}(\nu)$. 
In particular, they are both well-suited to converge 
to nonsmooth minimizers of $\alpha_{\epsilon}(\nu)$
provided that $\alpha_{\epsilon}(\nu)$ is differentiable at the iterates of these algorithms.
We refer to \cite{LewO2013} and \cite[Theorem 8.1]{MenYK14} for the convergence theory of 
{\GRANSO} and {\EIGOPT}, respectively, when the objective is nonsmooth at the minimizer.

In the small-scale setting, we minimize $\alpha_{\epsilon}(\nu)$ 
either with {\GRANSO} if there are many 
parameters, or with {\EIGOPT} if there are only a few parameters. 
Note that {\EIGOPT} at termination guarantees that the $\epsilon$-pseudospectral
abscissa objective at the computed estimate for the minimizer does not differ from the actual
globally smallest value of the objective by more than a prescribed tolerance.
On the other hand, for {\GRANSO} there are several termination criteria, but in
our context here it normally terminates either when the norm of the search direction is
less than a prescribed tolerance, or the line-search in the search direction fails 
to find a point satisfying the weak Wolfe conditions (i.e. that consist of a sufficient decrease and
a sufficient curvature condition \cite{Nocedal2006}). Both methods require not only
the computation of $\alpha_{\epsilon}(\nu)$, but also the first derivatives of $\alpha_{\epsilon}(\nu)$.
Fortunately, once $\alpha_{\epsilon}(\nu)$ as well as the rightmost point where it is attained are computed
by Algorithm~\ref{alg:criss-cross}, then the first derivatives of $\alpha_{\epsilon}(\nu)$
can be obtained at almost no additional computational cost.

To see this, for $P(\nu,s) := P(s_1 + {\rm i} s_2 ; \nu)$ and 
$\sigma(\nu,s) := \sigma_{\min}(P(\nu,s))$ with $s = (s_1, s_2) \in {\mathbb R}^2$, 
the $\epsilon$-pseudospectral abscissa function under consideration can be expressed
as 
\begin{equation}\label{eq:psa_opt}
	\alpha_{\epsilon}(\nu)
		\;	=	\;	
	\max \{ s_1	\;	|	\;
				s = (s_1, s_2) \in {\mathbb R}^2		\;	\text{ is s.t. }	\;
				\sigma(\nu,s) - \epsilon p_w( \| s \| )	\leq 0	\}
\end{equation}
with $\| s \| := \sqrt{s_1^2 + s_2^2}$ denoting the Euclidean norm of $s$, and
an associated Lagrangian function (in the primal variable $s$ and the dual variable $\mu$) is given by
\[
{\mathcal L}(\nu,s,\mu)	\;	:=	\;	s_1	-	\mu \{ \sigma(\nu,s) - \epsilon p_w( \| s \|) \}	\;	.
\]
For simplicity, we combine the primal and dual variable into a vector $y = (s,\mu)$.
Moreover, we denote by ${\mathcal L}_{s_1}$, ${\mathcal L}_{s_2}$ 
and ${\mathcal L}_{y}$, ${\mathcal L}_{yy}$ the partial derivative of ${\mathcal L}$ with
respect to $s_1$, $s_2$ and the gradient, the Hessian of ${\mathcal L}$ with respect to $y$.
If the optimizer (i.e. global maximizer) $s(\nu)$
of (\ref{eq:psa_opt}) is unique and $\sigma(\nu, s(\nu))$ is simple, then, 
assuming the satisfaction of a constraint qualification, there exists a
unique $\mu(\nu)$ such that ${\mathcal L}_y(\nu,s(\nu),\mu(\nu)) = 0$;
see the proof of Theorem \ref{thm:poly_psa_der} below.
We also combine the optimal values of the primal and dual variables, and write
$y(\nu) = (s(\nu),\mu(\nu))$.

The subsequent derivations make use of the smoothness result 
concerning the singular value functions given below. This can be regarded as a 
special case of the classical results related to the eigenvalue functions of Hermitian
matrix-valued functions that depend on its parameters in a real-analytic manner
\cite{Kato1995, Lancaster1964, Rellich1969}. The extension of these results to 
singular value functions has been discussed in
\cite[Section 3.3]{MenYK14}.
 \begin{lemma}\label{lem:sval_der}
 Let $A : {\mathbb R}^\ell \rightarrow {\mathbb C}^{p\times q}$ be a matrix-valued 
 function such that its real part $(1/2)\{ A + \overline{A} \}$ as well as imaginary part 
 $(-{\rm i}/2)\{ A - \overline{A} \}$ are real-analytic, and the $j$th largest singular value 
 $\sigma_j(\mu)$ of $A(\mu)$ is simple and nonzero at $\mu = \widetilde{\mu}$. 
 Then $\sigma_j(\mu)$ is real-analytic at $\mu = \widetilde{\mu}$, and its first 
 derivatives are given by
 \begin{equation}\label{eq:sval_der}
 	\frac{\partial \sigma_j}{\partial \mu_k} (\widetilde{\mu})
		\;	=	\;
	{\mathrm Re} \left\{ u^\ast_j  \, \frac{\partial A}{\partial \mu_k}(\widetilde{\mu}) \, v_j \right\}	
 \end{equation}
 for $k = 1, \dots , \ell$,
 where $u_j, v_j$ denote consistent unit left, respectively right singular vectors of $A(\widetilde{\mu})$ 
 corresponding to the singular value $\sigma_j(\widetilde{\mu})$.
 \end{lemma}
 
We refer to a point $\nu \in \Omega$ as a \emph{non-degenerate point} 
of \eqref{eq:psa_opt} if
\begin{enumerate}
	\item there is a unique optimizer of (\ref{eq:psa_opt}), 
	denoted as $s(\nu) = (s_1(\nu) , s_2(\nu))$,
	\item the singular value $\sigma(\nu, s(\nu))$ of $P(\nu,s(\nu))$ is simple, and
	\item the Hessian ${\mathcal L}_{yy}(\nu,y(\nu))$ is invertible. 
\end{enumerate}
As we will prove in the next result, the three conditions above combined with
a mild assumption guarantee the smoothness of $\alpha_{\epsilon}(\nu)$.
For the subsequent result, $P'(s ; \nu)$ stands for the ordinary derivative 
of the polynomial $P(\, \cdot \, ; \nu)$ given by 
$P'(s; \nu) = 2s M(\nu) + C(\nu)$. Additionally, $g$ stands 
for the function defined in (\ref{eq;defn_g}), but to make its dependence on the 
parameters $\nu$ explicit, we use
\begin{equation}\label{eq:defn_g2}
	g(s; \nu)  \; := \;   \sigma(\nu,s) - \epsilon p_w( \| s \|).
\end{equation}
Also, $g_{s_1}$, $g_{s_2}$ represent the partial derivatives of $g$ with respect
to $s_1$, $s_2$, and $g_s$, $g_{ss}$ the gradient of $g$, respectively the Hessian of $g$
with respect to $s$.

Before presenting the result, we remark on the third condition 
for the non-degeneracy of $\nu$, which is in terms of the 
invertibility of the Hessian ${\mathcal L}_{yy}(\nu,y(\nu))$. Using the
definition of ${\mathcal L}(\nu,y)$, it is immediate that
\[
	{\mathcal L}_{yy}(\nu,y(\nu))
			=	
	\left[
	\begin{array}{cc}
		-\mu(\nu) g_{ss}(s(\nu); \nu) 	&	-g_s(s(\nu); \nu)			\\	
		-g_s(s(\nu); \nu)^T		&			0
	\end{array}
	\right]	.
\]
Assuming that the first two conditions for the non-degeneracy of $\nu$ hold
and that $g_s(s(\nu); \nu) \neq 0$, we must have $\mu(\nu) \neq 0$
(as apparent from the proof of Theorem \ref{thm:poly_psa_der} below,
more specifically from (\ref{eq:KKT_poly_psa})). Then, if $g_{ss}(s(\nu); \nu)$
is invertible, the Hessian ${\mathcal L}_{yy}(\nu,y(\nu))$ must be invertible,
due to the fact that
\[
	\left[
	\begin{array}{cc}
		-\mu(\nu) g_{ss}(s(\nu); \nu) 	&	-g_s(s(\nu); \nu)			\\	
			0		&			(1/\mu(\nu)) g_s(s(\nu); \nu)^T g_{ss}(s(\nu); \nu)^{-1} g_s(s(\nu); \nu)
	\end{array}
	\right]	
\]  
is invertible, recalling that  $g_s(s(\nu); \nu) \neq 0$, $\mu(\nu) \neq 0$.
Hence, the first two conditions for the non-degeneracy of $\nu$ combined
with the assumptions (i) $g_{ss}(s(\nu); \nu)$ is invertible, (ii) $g_s(s(\nu); \nu) \neq 0$
imply the invertibility of ${\mathcal L}_{yy}(\nu,y(\nu))$.

\begin{theorem}\label{thm:poly_psa_der}
Let $\widetilde{\nu}$ be a non-degenerate point of \eqref{eq:psa_opt}. 
Moreover, suppose that
$p_w( \| s(\widetilde{\nu}) \|) \neq 0 \,$ and $\, g_s(s(\widetilde{\nu}); \widetilde{\nu}) \neq 0 \,$.
Then the function $\nu \mapsto \alpha_{\epsilon}(\nu)$
is real analytic at $\widetilde{\nu}$ with the first derivatives given by
\begin{equation}\label{eq:psa_der}
	\frac{\partial \alpha_{\epsilon}}{\partial \nu_j} (\widetilde{\nu})
		\;	=	\;
	-\mu(\widetilde{\nu})
	\mre \left\{ u^\ast P_{\nu_j}(\widetilde{\nu} , s(\widetilde{\nu})) v \right\}
	\,	,
	\;\;\;\;
	j = 1, \dots , \sd		\,	,
\end{equation}
where $P_{\nu_j}$ denotes the partial derivative of $P$ with respect to $\nu_j$,
\begin{equation}\label{eq:lagrange_mult}
	\mu(\widetilde{\nu})
		\;	:=	\;
	\frac{1}{\mre\{ u^\ast P'(s(\widetilde{\nu}) ; \widetilde{\nu}) v \} - 
				\epsilon \{ s_1(\widetilde{\nu}) / \| s(\widetilde{\nu}) \| \}  p'_w( \| s(\widetilde{\nu}) \|)}	\,	,
\end{equation}
and
$u, v$ are consistent unit left, right singular vectors, respectively, 
of $P(\widetilde{\nu} , s(\widetilde{\nu}))$ 
corresponding to its smallest singular value $\sigma(\widetilde{\nu} , s(\widetilde{\nu}))$.
\end{theorem}
\begin{proof}
As $\widetilde{\nu}$ is non-degenerate, 
(\ref{eq:psa_opt}) with $\nu = \widetilde{\nu}$ has a unique optimizer, 
which we denote by $s(\widetilde{\nu})$,  and $\sigma(\widetilde{\nu},s(\widetilde{\nu}))$ 
is simple. If $\sigma(\widetilde{\nu},s(\widetilde{\nu})) = 0$ holds, then, 
by the continuity of $\sigma(\widetilde{\nu},s)$ with respect to $s$,
there is a ball 
${\mathcal B} := \{ s \in {\mathbb R}^2 \; | \;  \| s - s(\widetilde{\nu}) \| \leq r \}$ 
with a positive radius $r$ centered at $s(\widetilde{\nu})$ such that
$\sigma(\widetilde{\nu},s) \leq \epsilon$ for all $s \in {\mathcal B}$, 
which contradicts  the optimality of $s(\widetilde{\nu})$
for (\ref{eq:psa_opt}) with $\nu = \widetilde{\nu}$.
Consequently, $\sigma(\widetilde{\nu},s(\widetilde{\nu}))$ is simple and nonzero.
By continuity, there is an open neighborhood $\widetilde{U} \subseteq {\mathbb R}^{\sd}$
of $\widetilde{\nu}$ and an open neighborhood $\widehat{U}  \subseteq {\mathbb R}^2$
of $s(\widetilde{\nu})$ such that
the singular value function $\sigma(\nu, s)$ remains 
simple and positive in $\widetilde{U} \times \widehat{U}$, while $p_w( \| s \|) > 0$ in $\widehat{U}$,
so $\sigma(\nu, s)$ as well as $g(s; \nu)$ defined as in (\ref{eq:defn_g2}) 
are real analytic in $\widetilde{U} \times \widehat{U}$ by Lemma \ref{lem:sval_der}.
We assume, without loss of generality, that $g_s(s;\nu) \neq 0$, indeed
the existence of a positive real number $\gamma$ such that $\| g_s(s;\nu)  \| \geq \gamma$,
for all $(s,\nu) \in \widetilde{U} \times \widehat{U}$ (if necessary by choosing 
$\widetilde{U}$, $\widehat{U}$ smaller).

The set of optimal solutions $S(\nu)$ of (\ref{eq:psa_opt}) is upper
semicontinuous in $\widetilde{U}$ \cite[Proposition 4.41]{BonS00} 
(i.e. for every neighborhood ${\mathcal S}$ 
of $S(\nu)$ there is a neighborhood ${\mathcal V}$ of $\nu \in \widetilde{U}$
such that ${\mathcal S} \supseteq S(\vartheta)$ 
for all $\vartheta \in {\mathcal V}$). In particular, there is a neighborhood $U \subseteq \widetilde{U}$
of $\widetilde{\nu}$ such that $S(\vartheta) \subseteq \widehat{U}$ for all $\vartheta \in U$. Let
$s(\vartheta)$ denote a maximizer of (\ref{eq:psa_opt}) for $\vartheta \in U$. (It is soon to be established
that (\ref{eq:psa_opt}) with $\vartheta \in U$ has a unique maximizer, so the notation $s(\vartheta)$ is 
consistent with the notation introduced before the theorem.) 
The singular value function $\sigma(\nu, s)$ is real analytic at 
$(\vartheta, s(\vartheta)) \in U \times \widehat{U} \subseteq \widetilde{U} \times \widehat{U}$ 
and $g_s(s(\vartheta);\vartheta) \neq 0$ from the previous paragraph.
The linear independence constraint qualification holds at $(\vartheta, s(\vartheta))$
due to $g_s(s(\vartheta);\vartheta) \neq 0$. Hence, by applying the first-order necessary 
conditions for a constrained optimization problem \cite[Theorem 12.1]{Nocedal2006}
to (\ref{eq:psa_opt}) with $\nu = \vartheta$, there is a Lagrange multiplier $\mu(\vartheta)$ satisfying
\begin{equation}\label{eq:KKT_poly_psa}
	 {\mathcal L}_{s_1}(\vartheta,s(\vartheta), \mu(\vartheta))
		\;	 =	\;
	1	-	
	\mu(\vartheta) \left\{ \sigma_{s_1}(\vartheta , s(\vartheta))
			- \epsilon \frac{\partial}{\partial s_1} \left[ p_w (\| s \| ) \right] 
					\bigg|_{s = s(\vartheta)} \right\}
			=	\;		0	\,	.
\end{equation}
It is evident from (\ref{eq:KKT_poly_psa}) that the Lagrange multiplier $\mu(\vartheta)$
corresponding to the maximizer $s(\vartheta)$ is unique. Letting, for $\nu \in U$,
$y(\nu) := (s(\nu) , \mu(\nu))$,
indeed, the first-order necessary conditions imply ${\mathcal L}_{y}(\nu,y(\nu)) = 0$
for every $\nu \in U$. 
Note also that $s(\widetilde{\nu})$ must be the unique 
maximizer of (\ref{eq:psa_opt}) with $\nu = \widetilde{\nu}$ (i.e. by assumption,
$\widetilde{\nu}$ is a non-degenerate point,
in particular (\ref{eq:psa_opt}) with $\nu = \widetilde{\nu}$ has a unique maximizer).

On the other hand, since ${\mathcal L}_{yy}(\widetilde{\nu},y(\widetilde{\nu}))$ 
is invertible by the non-degeneracy of $\widetilde{\nu}$,
the analytic implicit function theorem implies the existence of a 
neighborhood $U_0 \subseteq U$ of $\widetilde{\nu}$ and a unique function 
$w : U_0 \rightarrow \widehat{U}$ satisfying $w(\widetilde{\nu}) =  y(\widetilde{\nu})$
and ${\mathcal L}_{y}(\nu,w(\nu)) = 0$ for all $\nu \in U_0$. Moreover,
the function $w$ is real analytic. By the uniqueness of $w$, the function 
$\nu \in U_0 \mapsto y(\nu)$ is unique and $y(\nu) = w(\nu)$ for every $\nu \in U_0$. 
This also shows the real analyticity of $y(\nu)$ at every $\nu \in U_0$.
In particular, the map $\nu \in U \mapsto \alpha_{\epsilon}(\nu) = s_1(\nu)$ 
is real analytic at $\nu = \widetilde{\nu}$.

To derive the formulas for the derivative given by (\ref{eq:psa_der}) and (\ref{eq:lagrange_mult}), 
we note that
\begin{equation*}
	\begin{split}
	&
	\sigma_{s_1}(\widetilde{\nu} , s(\widetilde{\nu}))
			\;	=	\;
		\mre
		\left\{
			u^\ast
				\frac{\partial P}{\partial s_1}(\widetilde{\nu} , s(\widetilde{\nu}))
			v
		\right\}
			\;	=	\;
		\mre \{ u^\ast P'(s(\widetilde{\nu}) ; \widetilde{\nu}) v \}	\,	,	\\[.2em]
	&
		\frac{\partial}{\partial s_1} \left[ p_w (\| s \| ) \right] 
						\bigg|_{s = s(\widetilde{\nu})}
			\;	=	\;
		\frac{s_1(\widetilde{\nu})}{ \| s(\widetilde{\nu}) \| } \, p'_w ( \| s(\widetilde{\nu}) \|)	\;	,
	\end{split}
\end{equation*}
where, in the first line, we use the analytic formula in (\ref{eq:sval_der}),
in particular $u$, $v$ denote a pair of consistent unit left,
right singular vectors corresponding to the singular value 
$\sigma(\widetilde{\nu} , s(\widetilde{\nu}))$.
It then follows from (\ref{eq:KKT_poly_psa}) that
\[
	\mu(\widetilde{\nu})
		\;	=	\;
	\frac{1}{\mre \{ u^\ast P'(s(\widetilde{\nu}) ; \widetilde{\nu}) v \}
			-  \epsilon \{ s_1(\widetilde{\nu}) / \| s(\widetilde{\nu}) \| \} p'_w ( \| s(\widetilde{\nu}) \|) }	\,	.
\]
Since $\mu(\nu) \neq 0$ for $\nu \in U_0$ due to (\ref{eq:KKT_poly_psa}),
by the complementary conditions we must have $g(s(\nu); \nu) = 0$  for $\nu \in U_0$ 
implying that $\alpha_{\epsilon}(\nu) = s_1(\nu) = {\mathcal L}(\nu,y(\nu))$ for all $\nu \in U_0$.
Finally, by also making use of ${\mathcal L}_y (\widetilde{\nu}, y(\widetilde{\nu})) = 0$, 
we deduce that
\begin{equation*}
	\begin{split}
	\frac{\partial \alpha_{\epsilon}}{\partial \nu_j} (\widetilde{\nu})
			\;\;	=	\;
	\frac{\partial {\mathcal L}}{\partial \nu_j} (\widetilde{\nu},y(\widetilde{\nu}))
					+
{\mathcal L}_y (\widetilde{\nu}, y(\widetilde{\nu}))^T \frac{\partial y}{\partial \nu_j}(\widetilde{\nu})	
						\;	&  	 =	\;
	\frac{\partial {\mathcal L}}{\partial \nu_j} (\widetilde{\nu},y(\widetilde{\nu}))	 \\
			&	=	\;
			-\mu(\widetilde{\nu})
			\frac{\partial \sigma}{\partial \nu_j}(\widetilde{\nu} , s(\widetilde{\nu}))	\\[.3em]
			&	=	\;
			-\mu(\widetilde{\nu}) \mre \left\{ u^\ast P_{\nu_j}(\widetilde{\nu} , s(\widetilde{\nu})) v \right\}	\:	,
	\end{split}
\end{equation*}
where again the last equality follows from the formula in (\ref{eq:sval_der})
for the derivatives of a singular value function.
\end{proof}

\vskip -3ex

\begin{remark}
{\rm It is known in the matrix setting that the counterpart of the condition 
$g_s(s(\widetilde{\nu}); \widetilde{\nu}) \neq 0$ in Theorem \ref{thm:poly_psa_der}
always holds. In the matrix setting, the function 
$h(s; \nu) = \sigma_{\min}\left( A(\nu) - (s_1 + {\rm i} s_2) I \right) - \epsilon$ 
for a square matrix  $A(\nu)$ dependent on the parameters $\nu$ replaces $g(s; \nu)$.
Assuming that 
$\sigma_{\min}\left( A(\widetilde{\nu}) - (\widetilde{s}_1 + {\rm i} \widetilde{s}_2) I \right)$
is simple, nonzero at a given $\widetilde{\nu}$ and a locally rightmost point 
$\widetilde{s} = ( \widetilde{s}_1 , \widetilde{s}_2 )$ in the $\epsilon$-pseudospectrum
of $A(\widetilde{\nu})$, it follows from (\ref{eq:sval_der}) that
$h_s(\widetilde{s}; \widetilde{\nu}) = \left( -\text{Re}(u^\ast v) , \text{Im}(u^\ast v) \right)$,
where $u$,$v$ are consistent unit left, unit right singular vectors corresponding to
$\sigma_{\min}\left( A(\widetilde{\nu}) - (\widetilde{s}_1 + {\rm i} \widetilde{s}_2) I \right)$.
It has been argued in \cite[Lemma 2.6]{Guglielmi2011} as a corollary of 
\cite[Theorem 9]{AlaBBO11} that $u^\ast v \neq 0$\footnote{We thank Michael Overton
for pointing out these results in the matrix setting.}, 
so $h_s(\widetilde{s}; \widetilde{\nu}) \neq 0$.}
\end{remark}

\vskip -3ex

\begin{remark}{\rm
Suppose that $P(\lambda; \nu) = \lambda^2 M(\nu) + \lambda C(\nu) + K(\nu)$ 
is a Hermitian quadratic matrix polynomial for all $\nu \in \Omega$, i.e.  
$M(\nu), C(\nu), K(\nu)$ are Hermitian matrices for all $\nu \in \Omega$.
Then it is apparent from the characterization in (\ref{eq:char}) that
$\Lambda_{\epsilon}(P(\, \cdot \, ; \nu))$ is symmetric in the complex plane
with respect to the real axis, i.e. $z \in \Lambda_{\epsilon}(P(\, \cdot \, ; \nu))$ if and only if
$\, \overline{z} \in \Lambda_{\epsilon}(P(\, \cdot \, ; \nu))$, for all $\nu \in \Omega$. \\[.8em]
As a result, the optimization problem in (\ref{eq:psa_opt}) for 
$\alpha_{\epsilon}(P(\, \cdot \, ; \nu))$ can be restricted to
$(s_1, s_2) \in {\mathbb R}^2$ with $s_2 \geq 0$ rather than 
any $(s_1, s_2) \in {\mathbb R}^2$. The definition of a 
non-degenerate point $\nu \in \Omega$ can then be altered accordingly.
In particular, now at a non-degenerate point $\nu$ the optimization 
problem is required to have a unique optimizer $s(\nu) = (s_1(\nu) , s_2(\nu))$
in the upper half-plane $\{ (s_1, s_2) \in {\mathbb R}^2 \; | \; s_2 \geq 0 \} $
rather than ${\mathbb R}^2$. Theorem \ref{thm:poly_psa_der} in the
Hermitian setting still holds with these changes.}
\end{remark}

\section{Computation of the $\epsilon$-pseudospectral abscissa for large-scale problems}\label{sec:large_compute} 
For the large matrix setting, inspired by the ideas in \cite{KreV14}, 
a method for the computation of the $\epsilon$-pseudospectral abscissa of 
a nonlinear eigenvalue problem has been proposed in \cite{Mer17}. The approach
finds a locally  rightmost point in the $\epsilon$-pseudospectrum
for a projected nonlinear eigenvalue problem by means of the support-function based optimization 
technique in \cite{Men17}, and refines the projected nonlinear eigenvalue problem
by expanding the subspace based on the computed rightmost point. 

Even through the approach in \cite{Mer17} can be adapted to our polynomial eigenvalue 
problem setting,  it does not exploit the polynomial structure, in particular the tools such as 
the vertical and horizontal searches in the previous section.
Additionally, the approach in \cite{Mer17} is developed for the slightly different
definition of the $\epsilon$-pseudospectrum in \cite{Mic06,TisH01,TisM01} that uses the characterization  
$
	\{
		\lambda \in {\mathbb C}	\;	|	\;
			\sigma_{\min}(P(\lambda))		\leq	
			\epsilon  q_w (| \lambda |) 
	\}
$
for $q_w(z) = w_m z^2 + w_c z + w_k$. 

For these reasons we introduce a subspace
framework for our setting, motivated by the approaches in \cite{KreV14,Mer17}. Given a
subspace ${\mathcal V}$ of ${\mathbb C}^{\sn}$ and a matrix $V$ whose columns 
form an orthonormal basis for this subspace, i.e. $\mathcal V=\mbox{\rm Col}(V)$, the subspace framework to compute $\alpha_{\epsilon}(P)$,  replaces  $P(\lambda)$ by
the rectangular matrix polynomial
\[
	P^V(\lambda)
		\;	:=	\;
		P(\lambda) V
			\;	=	\;
		\lambda^2 (M V	)+	\lambda (C V)	+	(K V)		\:	.
\]
We define the associated $\epsilon$-pseudospectrum and
the $\epsilon$-pseudospectral abscissa of the projected problem by
\begin{equation}\label{eq:red_psa}
	\begin{split}
	\Lambda^{\mathcal V}_{\epsilon}(P)
		&	:=	\;	
			\{ 
			\lambda \in {\mathbb C}	\;	|	\;
			\sigma_{\min}(P^V(\lambda))		\leq	\epsilon p_w(| \lambda |) \}	\;	,
			\;\; \text{and}		\\[.3em]
	\alpha^{\mathcal V}_{\epsilon}(P)
		&	:=	\;
			\sup \{
					\mre (z)	\;	|	\;
					z \in {\mathbb C} \text{ s.t. }	
					\sigma_{\min}(P^V(z))	\leq \epsilon p_w(|z|) \}	\;	,
	\end{split}
\end{equation}
respectively. Clearly, the smallest singular value of $P^V(\lambda)$ depends
on the subspace ${\mathcal V}$, yet is
independent of the orthonormal basis $V$ that is used for ${\mathcal V}$.
Hence, as it is made explicit by the notations, the definitions of $\Lambda^{\mathcal V}_{\epsilon}(P)$
and $\alpha^{\mathcal V}_{\epsilon}(P)$ are independent of the choice of the orthonormal basis $V$
for ${\mathcal V}$.

The matrix polynomial $P^V(\lambda)$ is of size $\sn \times r$, where $r$ is the dimension
of the subspace ${\mathcal V}$. The associated $\epsilon$-pseudospectrum and the 
$\epsilon$-pseudospectral abscissa can be equivalently defined in terms of a reduced matrix
polynomial. In particular, let 
\[
	\left[
		\begin{array}{ccc}
			MV & CV & KV
		\end{array}
	\right]
		\;	=	\;
	Q
	\left[
		\begin{array}{ccc}
			\widehat{M}	&	\widehat{C}	&	\widehat{K}
		\end{array}
	\right]
\]
be a reduced QR factorization (see e.g. \cite{GolV96}) with $Q \in {\mathbb C}^{\sn \times 3r}$
and $\widehat{M}, \widehat{C}, \widehat{K} \in {\mathbb C}^{3r \times r}$. Then for the matrix polynomial
$\widehat{P}^V(\lambda) = \lambda^2 \widehat{M} + \lambda \widehat{C} + \widehat{K}$
of size $3r\times r$, we have
\[
		\sigma_{\min}(P^V(\lambda))	\;	=	\;
		\sigma_{\min}(\widehat{P}^V(\lambda))
		\quad	\mbox{\rm for all }\ \lambda \in {\mathbb C}.
\]
Hence, in the definitions of $\Lambda^{\mathcal V}_{\epsilon}(P)$
and $\alpha^{\mathcal V}_{\epsilon}(P)$ in (\ref{eq:red_psa}) the $3\sn \times r$
matrix polynomial $P^V$ can be replaced by the $3r\times r$ matrix polynomial $\widehat{P}^V$.

Algorithm \ref{alg:criss-cross} can be adapted to efficiently
compute $\alpha^{\mathcal V}_{\epsilon}(P)$, as well as a rightmost point in 
$\Lambda^{\mathcal V}_{\epsilon}(P)$. The only modification is that in the
vertical and horizontal searches, specifically in the definitions of ${\mathcal M}_j(x)$
and ${\mathcal N}_j(y)$ in Theorem \ref{thm:vert_src} and Theorem \ref{thm:horiz_src},
the matrix polynomial $P$ and matrix $M$ must be replaced by $\widehat{P}^V$
and $\widehat{M}$, respectively, and the sizes of the other zero, scalar-times-identity
blocks of ${\mathcal M}_j(x)$, ${\mathcal N}_j(y)$ must be adjusted accordingly. This results in 
matrix polynomials ${\mathcal M}(\lambda; x)$ and ${\mathcal N}(\lambda; y)$ of size 
$4r\times 4r$ required by the reduced problems
rather than their counterparts of size $2\sn\times 2\sn$ for the full problem.

The 
following properties of the  $\epsilon$-pseudospectrum
and $\epsilon$-pseudospectral abscissa 
have been shown in Lemma 4.1 and Theorem 4.3 of \cite{Mer17}.
\begin{theorem}[Monotonicity]\label{thm:monotonicity}
Let $\, {\mathcal U}$, ${\mathcal V}$ be two subspaces of $\, {\mathbb C}^n$ such that
$\, {\mathcal U} \subseteq {\mathcal V}$. The following assertions hold:
\begin{enumerate}
	\item  
	$\Lambda^{\mathcal U}_{\epsilon}(P) 
	\; \subseteq \;
	\Lambda^{\mathcal V}_{\epsilon}(P)
	\; \subseteq \;
	\Lambda_{\epsilon}(P)$;
	\item
	$\alpha^{\mathcal U}_{\epsilon}(P) 
	\; \leq \;
	\alpha^{\mathcal V}_{\epsilon}(P)
	\; \leq \;
	\alpha_{\epsilon}(P)$.
\end{enumerate}
\end{theorem}
\begin{theorem}[Low Dimensionality]\label{thm:low_dimensionality}
The following assertions are equivalent for a subspace ${\mathcal V}$ of ${\mathbb C}^n$:
\begin{enumerate}
	\item $\alpha_{\epsilon}(P)	\;	=	\;	\alpha^{\mathcal V}_{\epsilon}(P)$;
	\item The subspace ${\mathcal V}$ contains a vector $v$ that is a right
	singular vector corresponding to $\sigma_{\min}(P(\underline{z}))$ for some
	$\underline{z} \in \Lambda_{\epsilon}(P)$ such that $\mre(\underline{z}) = \alpha_{\epsilon}(P)$.
\end{enumerate}
\end{theorem}
The subspace framework to compute $\alpha_{\epsilon}(P)$ and a rightmost
point $\underline{z}$ in $\Lambda_{\epsilon}(P)$ when $\sn$ is large is presented
in Algorithm \ref{alg:sub_fw} below. Note that throughout this description
${\mathcal V}_k$ is the restriction subspace spanned by the columns of the matrix $V_k$.
The subspace framework gradually expands the subspace ${\mathcal V}_k$ 
until the correct vector or a nearby vector as in Theorem \ref{thm:low_dimensionality}
is in this subspace, ensuring that $\alpha_{\epsilon}(P) \approx \alpha^{{\mathcal V}_k}_{\epsilon}(P)$.
At every iteration, a rightmost point $z^{(k)}$ in the reduced pseudospectrum
is computed by means of Algorithm~\ref{alg:criss-cross} in line \ref{line:solve_red}. 
If the condition for termination is not met, then the subspace is expanded by
including a right singular vector of $P(z^{(k)})$ corresponding to its smallest 
singular value in line \ref{line:expand}.

\begin{algorithm}
\begin{algorithmic}[1]
	\REQUIRE{Quadratic matrix polynomial $P(\lambda)$ of the form (\ref{eq:mat_poly}) with large $\sn$,
				and $\epsilon > 0$.}
	\ENSURE{Estimates $\underline{x}$ for $\alpha_{\epsilon}(P)$, and $\underline{z} \in {\mathbb C}$
				for a rightmost point in $\Lambda_{\epsilon}(P)$.}	
	\vskip 1.2ex
	\STATE{$z^{(0)} \gets$ an eigenvalue of $P(\lambda)$.}\label{alg:init_points2_sf}
	\vskip .4ex
	\STATE{$V_0 \gets$ a unit norm right singular vector
								corresponding to $\sigma_{\min}(P(z^{(0)}))$.}\label{line:init_sub}
	\vskip 1.7ex
	\FOR{$k=1,2,\dots$}
	\vskip 1.7ex
	\STATE{\textbf{Solve Reduced Problem.} \\[.2em]
		Find a rightmost point  $z^{(k)}$ in $\Lambda^{{\mathcal V}_{k-1}}_{\epsilon}(P)$
		using Algorithm~\ref{alg:criss-cross}.  
	}\label{line:solve_red} 
	\vskip 1.7ex	
	\STATE{\textbf{Return} $\underline{x} \gets \text{Re}(z^{(k)}) \,$ and
	$\, \underline{z} \gets z^{(k)} \;$  if convergence occurred.} \label{alg:terminate_sf}
	\vskip 1.7ex
	\STATE{\textbf{Expand Subspace.} \\[.2em]
		Compute a right singular vector $v$
		corresponding to $\sigma_{\min}(P(z^{(k)}))$,
		and set \\[.2em]
		$V_k \gets 
			\text{orth}
			\left(
				\left[
					\begin{array}{cc}
						V_{k-1}	&	v
					\end{array}
				\right]
			\right)$.}\label{line:expand}
	\vskip 1.5ex
	\ENDFOR
\end{algorithmic}
\caption{Subspace framework for large-scale computation of $\alpha_{\epsilon}(P)$}
\label{alg:sub_fw}
\end{algorithm}
%
 \textbf{Convergence.}
	A variant of the subspace framework for the matrix setting in \cite{KreV14}
	converges to a locally rightmost point in the $\epsilon$-pseudospectrum. 
	We observe a similar local convergence behavior for Algorithm~\ref{alg:sub_fw}
	in the quadratic matrix polynomial setting. The choice of the
	initial projection subspace has an important role in attaining global convergence.
	To this end, the initialization in line \ref{line:init_sub} can possibly be done in a
	more sophisticated way. Additionally, it has been shown in \cite{KreV14} that
	$\text{Re}(z^{(k)})$ converges as $k \rightarrow \infty$ rather quickly at a superlinear rate. 
	Such a quick convergence is expected in our polynomial eigenvalue setting as well
	due to the interpolation properties that are indicated next.

\textbf{Interpolation.}
	Suppose that the computed rightmost point $z^{(\ell)}$ after the execution
	of line \ref{line:solve_red} at the $\ell$th iteration, $\ell \geq 1$,
	is such that $\sigma_{\min}(P(z^{(\ell)}))$ is 
	simple and nonzero.
	Using the variational characterization of the smallest singular value \cite{GolV96}, as well
	as the analytical formulas for the derivatives of singular value functions
	(in particular Lemma \ref{lem:sval_der}),
	it can be shown for Algorithm \ref{alg:sub_fw} that
	the Hermite interpolation properties
	\begin{equation}\label{eq:psa_comp_HI}
		\begin{split}
		&
		\sigma_{\min}(P(z^{(\ell)}))	=	\sigma_{\min}(P^{V_k}(z^{(\ell)}))	\,	,	\\[.2em]
		&
		\frac{\partial }{\partial \, \mre(z)} [\sigma_{\min}(P(z))]	\bigg|_{z = z^{(\ell)}}
				=
		\frac{\partial }{\partial \, \mre(z)} [\sigma_{\min}(P^{V_k}(z))]	\bigg|_{z = z^{(\ell)}}	\,	, \\[.2em]
		&
		\frac{\partial }{\partial \, \mri(z)} [\sigma_{\min}(P(z))]	\bigg|_{z = z^{(\ell)}}
				=
		\frac{\partial }{\partial \, \mri(z)} [\sigma_{\min}(P^{V_k}(z))]	\bigg|_{z = z^{(\ell)}}	\,
		\end{split}
	\end{equation}
	are satisfied for every $k \geq \ell$. Note that the interpolation of the singular values 
	$\sigma_{\min}(P(z^{(\ell)})) = \sigma_{\min}(P^{V_k}(z^{(\ell)}))$ holds 
	even if $\sigma_{\min}(P(z^{(\ell)}))$ is not simple or is zero.	
	
\textbf{Termination.}
	In practice, we terminate the algorithm in line \ref{alg:terminate_sf} when the
	change in $\underline{x}$ in two consecutive iterations is less than
	a tolerance in a relative sense.

\textbf{Large-scale requirements and computational cost}.
	The only large-scale problems whose solutions are required by Algorithm \ref{alg:sub_fw}
	are the computation of an eigenvalue of $P(\lambda)$ for initialization,
	and the singular vector computation corresponding to $\sigma_{\min}(P(z^{(0)}))$
	in line \ref{line:init_sub}, as well as corresponding to
	$\sigma_{\min}(P(z^{(k)}))$ in line \ref{line:expand} at every iteration.
	Both of these computations can be carried out using iterative methods 
	such as an appropriate Krylov subspace method \cite{arpack}, e.g. Arnoldi or Lanczos 
	method. In practice, we use \textsf{eigs}, an implementation of an implicitly restarted 
	Arnoldi method, in MATLAB. Krylov subspace methods are especially well-suited for large
	sparse eigenvalue and singular value problems, even though they may converge 
	to eigenvalues, singular values other than the targeted ones, or exhibit very slow 
	convergence in some cases. Clearly, reliability of Algorithm \ref{alg:sub_fw}
	relies on the reliability of Krylov subspace methods.
	In our setting here, 
     they usually work quite 
	effectively and reliably especially on problems arising from damping optimization.
	The most expensive computational task at every iteration of Algorithm \ref{alg:sub_fw}
	is the smallest singular value computation in line \ref{line:expand}. By using
	a Krylov subspace method, this usually has computational cost 
	${\mathcal O}(\sn^2)$ (possibly ${\mathcal O}(\sn)$ if the matrix polynomial
	is sparse), but it may be higher especially
	if the convergence of the Krylov subspace method is slow. Assuming a
	superlinear convergence of Algorithm \ref{alg:sub_fw}, the Krylov subspace
	method to compute the smallest singular value in line \ref{line:expand} is called 
	only a few times. As a result, it is reasonable to assume that the overall cost of
	Algorithm \ref{alg:sub_fw} is 
	${\mathcal O}(\sn^2)$, but possibly with a big constant hidden
	in the ${\mathcal O}$-notation.

\section{Large-Scale minimization of the $\epsilon$-pseudospectral abscissa}\label{sec:large_minimize} 
In this section, we discuss the minimization of $\alpha_{\epsilon}(P(\, \cdot \, ; \nu))$
over $\nu \in \underline{\Omega}$ in the large-scale setting, i.e.  when $\sn$ is large.
Similar to the ideas in the previous section, but now for minimization and
not for computation, the large parameter-dependent matrix polynomial 
$P(\, \lambda \, ; \nu)$ is replaced by
\[
	P^V(\lambda ; \nu)
		\;	:=	\;
	P(\lambda ; \nu) V
		\;	=	\;
	\lambda^2 M^V(\nu)	+	\lambda C^V(\nu)  + K^V(\nu) \, ,
\]
where
\begin{equation*}
	\begin{split}
	&
	M^V(\nu)
		\;\;	:=	\;\;
	M(\nu) V	
	\;	=	\;\:	\psi_{1}(\nu) (M_{1} V) +	\dots		+	\psi_{\sm}(\nu)	(M_{\sm} V)	\,	,	\\[.3em]
	& \,
	C^V(\nu)
		\;\;	:=	\;\;
	C(\nu) V	
\;\:	=	\;\:	\zeta_{1}(\nu) (C_{1} V) +	\dots		+	\zeta_{\scc}(\nu)	 (C_{\scc} V)	\,	,	\\[.3em]
	& \,
	K^V(\nu)
		\;\;	:=	\;\;
	K(\nu)	 V
	\;\,	=	\;\:	\kappa_{1}(\nu) (K_{1} V) +	\dots		+	\kappa_{\sk}(\nu)	 (K_{\sk} V)	\,	
	\end{split}
\end{equation*}
for a matrix $V \in {\mathbb C}^{\sn \times r}$ whose columns form an orthonormal basis 
for a small $r \ll \sn$ dimensional restriction subspace ${\mathcal V}$. We form 
$M_{i}V, \, C_{j}V, \, K_{\ell}V \in {\mathbb C}^{\sn\times r}$ for $i = 1, \dots , \sm$,
$j = 1, \dots , \scc$, $\: \ell = 1, \dots , \sk \,$ in advance, then minimize 
\begin{equation}\label{eq:red_obj_pars}
	\alpha^{\mathcal V}_{\epsilon}(P(\, \cdot \, ;\nu))
		\;	:=	\;
			\sup \{
					\mre (z)	\;	|	\;
					z \in {\mathbb C} \text{ s.t. }	
					\sigma_{\min}(P^V(z; \nu))	\leq \epsilon p_w(|z|) \}
\end{equation}
over $\nu \in \underline{\Omega}$, which we refer to as a \emph{reduced problem}. 
The value of $\alpha^{\mathcal V}_{\epsilon}(P(\, \cdot \,,\nu))$
at a given $\nu$, as well as the corresponding $z$ maximizing the right-hand side in (\ref{eq:red_obj_pars})
can be computed efficiently using a variant of Algorithm~\ref{alg:criss-cross} as discussed 
in Section \ref{sec:large_compute}.

We carry out the minimization of $\alpha_{\epsilon}(P(\, \cdot \, ; \nu))$ over 
$\nu \in \underline{\Omega}$ in the large-scale setting using the subspace restriction 
idea as follows. 
At every iteration, we solve a reduced $\epsilon$-pseudospectral abscissa minimization problem. 
Then a right-most point in $\Lambda_{\epsilon}(P(:,\nu))$ 
is computed at the minimizing parameter value for the reduced problem. 
Finally, we enrich the restriction subspace by a right
singular vector corresponding to $\sigma_{\min}(P(z; \nu))$ at
the minimizing parameter value $\nu$ and the computed
right-most point $z$. 

This subspace approach is formally presented in Algorithm \ref{alg:sub_fw_min}, with
${\mathcal V}_k := \text{Col}(V_k)$, and the iterates $\nu^{(k)}, z^{(k)}$ at the $k$th step.

\begin{algorithm}
\begin{algorithmic}[1]
	\REQUIRE{Quadratic matrix polynomial $P(\lambda; \nu)$ of the form (\ref{eq:par_dep_mat_poly}) 
	with large $\sn$, the feasible region $\underline{\Omega} \,$, and $\epsilon > 0$.}
	\ENSURE{Estimates $\underline{\nu}$ for 
	$\arg\min_{\nu\in\underline{\Omega}}\alpha_\epsilon (P(\, \cdot \, ; \nu))$, 
			and $\underline{z} \in {\mathbb C}$
				 for a globally rightmost point in $\Lambda_{\epsilon}(P(\, \cdot \, ; \underline{\nu}))$.}	
	\vskip 2.1ex
	\STATE{\textbf{Initialize.} \\[.2em]
	$\nu^{(0)}_1, \dots \nu^{(0)}_\eta \gets$ initially chosen points in $\underline{\Omega}$. \\[.2em]
	$z^{(0)}_j \gets$
			a rightmost point in $\Lambda_{\epsilon}(P(\, \cdot \, ; \nu^{(0)}_j))$
			using Algorithm  \ref{alg:sub_fw} for $j = 1, \dots , \eta$. \label{line:absk_sf_min0} \\[.2em]
	\hskip -.2ex $v^{(0)}_{j} \gets$ a right singular vector 
					corresponding to $\sigma_{\min}(P(z^{(0)}_j; \nu^{(0)}_j))$ 
					for $j = 1, \dots, \eta$. \\[.2em]	
	$V_0 \gets$ a matrix whose columns form an orthonormal basis for 
					span$\{v^{(0)}_{1} , \dots v^{(0)}_{\eta} \}$.}\label{defn:V0_sf_min} 
	\vskip 2.3ex
	\FOR{$k=1,2,\dots$}
	\vskip 2.1ex
	\STATE{\textbf{Solve Reduced Minimization Problem.} \\[.2em]
			$\nu^{(k)}\gets\arg\min_{\nu\in\underline{\Omega}}
				\alpha_\epsilon^{{\mathcal V}_{k-1}} (P(\, \cdot \, ; \nu))$.} \label{siter_start_sf_min}
	\vskip 2.3ex
	\STATE{$z^{(k)} \gets$ a rightmost point in $\Lambda_{\epsilon}(P(\, \cdot \, ; \nu^{(k)}))$
			using Algorithm  \ref{alg:sub_fw}.}\label{line:absk_sf_min}
	\vskip 2.3ex
	\STATE{\textbf{Return} $\underline{\nu} \gets \nu^{(k)}$, 
			$\underline{z} \gets z^{(k)}$ if convergence occurred.} \label{alg:terminate_sf_min}
	\vskip 2.1ex
	\STATE{\textbf{Expand Subspace.} \\[.2em]
		Compute a right singular vector $v^{(k)}$
		corresponding to $\sigma_{\min}(P(z^{(k)}; \nu^{(k)}))$, \\[.2em]
		and set 
		$V_k \gets 
			\text{orth}
			\left(
				\left[
					\begin{array}{cc}
						V_{k-1}	&	v^{(k)}
					\end{array}
				\right]
			\right)$.}\label{line:expand_sf_min}	
	\vskip 2.1ex
	\ENDFOR
\end{algorithmic}
\caption{Subspace framework for large-scale minimization of 
$\alpha_{\epsilon}(P(\, \cdot \, ;\nu))$}
\label{alg:sub_fw_min}
\end{algorithm}

In practice, we terminate Algorithm \ref{alg:sub_fw_min} in line~\ref{alg:terminate_sf_min} 
when the optimal values of the reduced minimization problems in two consecutive
iterations do not differ by more than a prescribed
tolerance. A remarkable feature of Algorithm \ref{alg:sub_fw_min} is that
it does not require the computation of all eigenvalues of large matrix polynomials. 
The only required large-scale computations are 
the computation of the right singular vector corresponding to the smallest 
singular value of large matrices in lines \ref{defn:V0_sf_min}, \ref{line:expand_sf_min},
and in Algorithm~\ref{alg:sub_fw}, as well as
the computation of an eigenvalue of a large matrix polynomial to
initialize Algorithm~\ref{alg:sub_fw}.
In practice, we perform these smallest singular value and eigenvalue computations
by means of Krylov subspace methods  as noted at the end of Section \ref{sec:large_compute}. The limitations of the Krylov subspace methods listed
there apply, but in practice we usually observe good convergence behavior  which is due to the fact that in the context of the optimization we usually have very good starting values.

The main ingredients that determine the runtime of Algorithm \ref{alg:sub_fw_min} are
\vskip 1ex
\begin{enumerate}
	\item[(i)$\:$] the solution of the reduced minimization problem in line \ref{siter_start_sf_min}, 
	\item[(ii)$\,$] the $\epsilon$-pseudospectral abscissa computation on the full 
	problem via Algorithm~\ref{alg:sub_fw} in line \ref{line:absk_sf_min}, 
	\item[(iii)]  the smallest singular value computation on the full problem in line
	\ref{line:expand_sf_min}
\end{enumerate}
\vskip 1ex
at every iteration. As argued in Section \ref{sec:large_compute}, it is reasonable to assume
both (ii) and (iii) have ${\mathcal O}(\sn^2)$ computational cost. However,
Algorithm~\ref{alg:sub_fw} itself computes the smallest singular value a few times, 
so (ii) is usually more expensive than (iii). Thus, the overall cost of Algorithm \ref{alg:sub_fw_min} 
is mainly determined by the cost of (i) and (ii) times the number of iterations of
Algorithm \ref{alg:sub_fw_min}. For large $\sn$, we expect (ii) to dominate, as (i)
involves only small reduced problems, that is the overall computational cost may
largely be determined by the number of calls to Algorithm~\ref{alg:sub_fw}.


\subsection{Solution of reduced minimization problems}
Just as for the full $\epsilon$-pseudospectral abscissa function,
the objective to be minimized $\alpha^{{\mathcal V}}_{\epsilon}(\nu) 
			:= \alpha^{{\mathcal V}}_{\epsilon}(P(\, \cdot \, ; \nu))$
for a reduced minimization problem is nonconvex and nonsmooth. Thus, 
$\alpha^{{\mathcal V}}_{\epsilon}(\nu)$ can again be minimized globally using {\EIGOPT}
if there are only a few parameters, or locally using {\GRANSO} if there are many parameters.
We again use the variant of Algorithm~\ref{alg:criss-cross} 
discussed in Section \ref{sec:large_compute} 
to evaluate the objective function $\alpha^{{\mathcal V}}_{\epsilon}(\nu)$,
and to determine a right-most point in 
$\Lambda^{{\mathcal V}}_{\epsilon}(P(\, \cdot \, ; \nu))$ from which
the gradient of $\alpha^{{\mathcal V}}_{\epsilon}(\nu)$ can be obtained as shown next. 
The arguments here are analogous to those for the full $\epsilon$-pseudospectral abscissa function.

The reduced $\epsilon$-pseudospectral abscissa function $\alpha^{{\mathcal V}}_{\epsilon}(\nu)$
can be written as
\begin{equation}\label{eq:psa_red_opt}
	\alpha^{\mathcal V}_{\epsilon}(\nu)
		\;	=	\;	
	\max \{ s_1	\;	|	\;
				s = (s_1, s_2) \in {\mathbb R}^2		\;	\text{ is s.t. }	\;
				\sigma^{\mathcal V}(\nu,s) - \epsilon p_w( \| s \|)	\leq 0	\}	\,	,
\end{equation}
where $\sigma^{\mathcal V}(\nu,s) := \sigma_{\min}(P^V(\nu,s))$, 
$P^V(\nu,s) := P^V(s_1 + {\rm i} s_2 ; \nu)$ for $s = (s_1, s_2) \in {\mathbb R}^2$ 
and for any matrix $V$
whose columns form an orthonormal basis for ${\mathcal V}$. The associated 
Lagrangian function is
\[
	{\mathcal L}^{\mathcal V}(\nu,s,\mu)	\;	:=	\;	s_1	-	
	\mu \{ \sigma^{\mathcal V}(\nu,s) - \epsilon p_w( \| s \|) \}	\,	,
\]
which we express as ${\mathcal L}^{\mathcal V}(\nu, y)$ where again
$y = (s, \mu)$. We reserve the notations ${\mathcal L}^{\mathcal V}_y$ and 
${\mathcal L}^{\mathcal V}_{yy}$
for the gradient and the Hessian of ${\mathcal L}^{\mathcal V}$ with respect to $y$.
Assuming the uniqueness of the maximizer $s^{\mathcal V}(\nu)$
for the maximization problem in (\ref{eq:psa_red_opt}), the simplicity of the singular value
$\sigma^{\mathcal V}(\nu,s^{\mathcal V}(\nu))$, and 
the satisfaction of a constraint qualification for (\ref{eq:psa_red_opt})
at $s = s^{\mathcal V}(\nu)$,
there is a unique $\mu^{\mathcal V}(\nu)$ satisfying
${\mathcal L}^{\mathcal V}_y(\nu,s^{\mathcal V}(\nu),\mu^{\mathcal V}(\nu)) = 0$. 
Similar to the practice before, we use 
$y^{\mathcal V}(\nu) := (s^{\mathcal V}(\nu), \mu^{\mathcal V}(\nu))$.

As in the non-reduced case, a point $\nu \in \Omega$ is called a non-degenerate point 
of the restriction of 
$P(\, \cdot \, ; \, \cdot \,)$ to a subspace ${\mathcal V}$ if there is a unique optimizer of (\ref{eq:psa_red_opt}), denoted as 
			$s^{\mathcal V}(\nu) = ( s^{\mathcal V}_1(\nu) , s^{\mathcal V}_2(\nu) )$,
the singular value $\sigma^{\mathcal V}(\nu,s^{\mathcal V}(\nu))$ 
	is simple, and
the Hessian 
		${\mathcal L}^{\mathcal V}_{yy}(\nu,y^{\mathcal V}(\nu))$
		is invertible.

The reduced problem satisfies an analogue of 
Theorem \ref{thm:poly_psa_der} but now
by operating on ${\mathcal L}^{\mathcal V}(\nu,y)$ rather than ${\mathcal L}(\nu,y)$.
Similar to the requirement in Theorem \ref{thm:poly_psa_der} on the gradient
$g_s$ of $g$ defined as in (\ref{eq:defn_g2}), we need a requirement here but now in terms of
\[
	g^{\mathcal V}(s; \nu)  \; := \;   \sigma^{\mathcal V}(\nu,s) - \epsilon p_w( \| s \|) ,
\]
in particular in terms of its gradient with respect to $s$, which we denote by
$g^{\mathcal V}_s$. We also use $g^{\mathcal V}_{ss}$ to denote the Hessian of $g^{\mathcal V}$
with respect to $s$. Note that, as in the non-reduced case, the first two conditions 
for the non-degeneracy of a point $\nu$ of the restriction of $P(\, \cdot \, ; \, \cdot \,)$ to ${\mathcal V}$
combined with the assumptions that $g^{\mathcal V}_{ss}(s^{\mathcal V}(\nu); \nu)$ is invertible
and $g^{\mathcal V}_s(s^{\mathcal V}(\nu); \nu) \neq 0$ imply the invertibility of 
${\mathcal L}^{\mathcal V}_{yy}(\nu,y^{\mathcal V}(\nu))$.
\begin{theorem}\label{thm:poly_red_psa_der}
Let ${\mathcal V}$ be a subspace of ${\mathbb C}^n$, and $V$ be any matrix whose
columns form an orthonormal basis for ${\mathcal V}$. 
Moreover, let $\, \widetilde{\nu}$ be a 
non-degenerate point for the restriction of $P(\, \cdot \, ; \, \cdot \,)$ to 
${\mathcal V}$, and suppose that $p_w( \| s^{\mathcal V}(\widetilde{\nu}) \|) \neq 0 \,$, 
$\, g^{\mathcal V}_s(s^{\mathcal V}(\widetilde{\nu}); \widetilde{\nu}) \neq 0$. 
The function $\nu \mapsto \alpha^{\mathcal V}_{\epsilon}(\nu)$
is real analytic at $\widetilde{\nu}$ with the first derivatives
\begin{equation}
\label{eq:red_psa_der}
	\frac{\partial \alpha^{\mathcal V}_{\epsilon}}{\partial \nu_j} (\widetilde{\nu})
		\;	=	\;
	-\mu^{\mathcal V}(\widetilde{\nu})
\mre \left\{ u^\ast P^V_{\nu_j}(\widetilde{\nu} , s^{\mathcal V}(\widetilde{\nu})) v \right\} \,	,
	\;\;\;\;
	j = 1, \dots , \sd		\,	,
\end{equation}
where $P^V_{\nu_j}$ denotes the partial derivative of $P^V$ with respect to $\nu_j$,
\begin{equation}\label{eq:red_lagrange_mult}
	\mu^{\mathcal V}(\widetilde{\nu})
		\;	:=	\;
	\frac{1}{\mre\{ u^\ast P'(s^{\mathcal V}(\widetilde{\nu}) ; \widetilde{\nu}) V v \} - 
	\epsilon \{ s^{\mathcal V}_1(\widetilde{\nu}) /  \| s^{\mathcal V}(\widetilde{\nu}) \| \}  
								p'_w( \| s^{\mathcal V}(\widetilde{\nu}) \|)}	\,	,
\end{equation}
and
$u, v$ are consistent unit left, right singular vectors, respectively, corresponding
 to the smallest singular value $\sigma^{\mathcal V}(\widetilde{\nu} , s^{\mathcal V}(\widetilde{\nu}))$ of 
$P^V(\widetilde{\nu} , s^{\mathcal V}(\widetilde{\nu}))$.
\end{theorem}
\begin{proof}
    The proof is similar to that for Theorem \ref{thm:poly_psa_der}.
\end{proof}
\subsection{Interpolation properties}
We now show that the subspace framework in Algorithm~\ref{alg:sub_fw_min} is
interpolatory to the full problem. 
Using the abbreviations
 $
	\Lambda_{\epsilon}(\nu)			:=		\Lambda_{\epsilon}(P(\, \cdot \, ; \nu))$	
		and	
$\Lambda^{{\mathcal V}_k}_{\epsilon}(\nu)			:=		
				\Lambda^{{\mathcal V}_k}_{\epsilon}(P(\, \cdot \, ; \nu))$,
we have the following result.
\begin{theorem}[Hermite Interpolation]\label{thm:hermite_interpolate}
Consider the subspaces and sequences generated by Algorithm \ref{alg:sub_fw_min}.
Then the following assertions hold for every $k \geq 1$ and $\ell = 1, \dots , k$:
\begin{enumerate}
	\item $\alpha_{\epsilon}(\nu^{(\ell)}) = \alpha^{{\mathcal V}_k}_{\epsilon}(\nu^{(\ell)})$;
	\item The point $z^{(\ell)}$ is a rightmost point in 
			$\Lambda^{{\mathcal V}_k}_{\epsilon}(\nu^{(\ell)})$; 
	\item Suppose $\nu^{(\ell)}$ is a non-degenerate point such that
	$p_w( \| \widehat{z}^{(\ell)} \|) \neq 0$, 
					$g_s(\widehat{z}^{(\ell)}; \nu^{(\ell)}) \neq 0$ for
	$\widehat{z}^{(\ell)} := (\text{Re}(z^{(\ell)}), \text{Im}(z^{(\ell)}))$,
	and the Hessian ${\mathcal L}^{{\mathcal V}_k}_{yy}
			(\nu^{(\ell)},\widehat{z}^{(\ell)}, \mu^{{\mathcal V}_k}(\nu^{(\ell)}))$
		is invertible, where
		$\mu^{{\mathcal V}_k}(\nu^{(\ell)})$ is as in (\ref{eq:red_lagrange_mult})
		but with $\nu^{(\ell)},\widehat{z}^{(\ell)}$ replacing 
		$\widetilde{\nu} , s^{\mathcal V}(\widetilde{\nu})$.
	Then $\alpha_{\epsilon}(\nu)$,
	$\alpha^{{\mathcal V}_k}_{\epsilon}(\nu)$ are differentiable at $\nu = \nu^{(\ell)}$,
	and $\nabla \alpha_{\epsilon}(\nu^{(\ell)}) = \nabla \alpha^{{\mathcal V}_k}_{\epsilon}(\nu^{(\ell)})$.
\end{enumerate} 
\end{theorem}
\begin{proof} 
1. The subspace ${\mathcal V}_k$ contains  a right singular vector $v^{(\ell)}$
	corresponding to $\sigma_{\min}(P(z^{(\ell)}; \nu^{(\ell)}))$, where $z^{(\ell)}$ is
	a rightmost point in $\Lambda_{\epsilon}(\nu^{(\ell)})$. Hence, the assertion
	$\alpha_{\epsilon}(\nu^{(\ell)}) = \alpha^{{\mathcal V}_k}_{\epsilon}(\nu^{(\ell)})$
	follows from Theorem \ref{thm:low_dimensionality}.
	
	\medskip
	
\noindent
2. By Theorem \ref{thm:monotonicity}, part 2., we have 
	\begin{equation}\label{eq:inter_ineq1}
		\alpha^{{\mathcal V}_k}_{\epsilon}(\nu^{(\ell)})
			\;	\leq	\;
		\alpha_{\epsilon}(\nu^{(\ell)}) = \mre (z^{(\ell)} )	\,	.
	\end{equation}
	Moreover, as $v^{(\ell)} \in {\mathcal V}_k$, there exists a vector $a$ such that $v^{(\ell)} = V_k a$.
	Combining this with $z^{(\ell)} \in \Lambda_{\epsilon}(\nu^{(\ell)})$, we obtain
\begin{equation*}
	\begin{split}
\epsilon	\;	\geq	\;	
\frac{\sigma_{\min}(P(z^{(\ell)};\nu^{(\ell)}))}{p_w( | z^{(\ell)} |)}
	\;	&	=	\;
\frac{\| P(z^{(\ell)};\nu^{(\ell)}) V_k a\|}{p_w( | z^{(\ell)} |)}	\\
	&	\geq		\;
	\frac{ \sigma_{\min}( P^{V_k}(z^{(\ell)};\nu^{(\ell)})) }{p_w( | z^{(\ell)} |)}	\,	,
\end{split}
\end{equation*}
where $\| \cdot \|$ denotes the Euclidean norm in ${\mathbb C}^{\sn}$.
This shows that  $z^{(\ell)} \in \Lambda^{{\mathcal V}_k}_{\epsilon}(\nu^{(\ell)})$, implying that
\begin{equation}\label{eq:inter_ineq2}
\alpha^{{\mathcal V}_k}_{\epsilon}(\nu^{(\ell)})
\;	\geq	\;
\alpha_{\epsilon}(\nu^{(\ell)}) = \mre (z^{(\ell)} )	\,	.	
\end{equation}
Combining (\ref{eq:inter_ineq1}) and (\ref{eq:inter_ineq2}),
	we deduce that $z^{(\ell)} \in \Lambda^{{\mathcal V}_k}_{\epsilon}(\nu^{(\ell)})$
	satisfies $\alpha^{{\mathcal V}_k}_{\epsilon}(\nu^{(\ell)})
			 = \mre (z^{(\ell)} )$ as claimed.

\medskip

\noindent			 
3. As $\nu^{(\ell)}$ is a non-degenerate point, 
and $g_s(\widehat{z}^{(\ell)}; \nu^{(\ell)}) \neq 0$ for
	$\widehat{z}^{(\ell)} := (\text{Re}(z^{(\ell)}), \text{Im}(z^{(\ell)}))$,
it follows from Theorem \ref{thm:poly_psa_der}
that $\alpha_{\epsilon}(\nu)$ is
	differentiable at $\nu = \nu^{(\ell)}$.
	
	
	We claim that $\alpha^{{\mathcal V}_k}_{\epsilon}(\nu)$ is also differentiable 
	at $\nu = \nu^{(\ell)}$. To show this, recall by part 2. that $z^{(\ell)}$ is a rightmost point of 
	$\Lambda^{{\mathcal V}_k}_{\epsilon}(\nu^{(\ell)})$. Suppose 
	there is another rightmost point $\widetilde{z}$ of $\Lambda^{{\mathcal V}_k}_{\epsilon}(\nu^{(\ell)})$
	satisfying $\widetilde{z} \in \Lambda^{{\mathcal V}_k}_{\epsilon}(\nu^{(\ell)})$
	and 
	\[
		\mre (\widetilde{z}) = \mre (z^{ (\ell)} ) =  \alpha^{{\mathcal V}_k}_{\epsilon}(\nu^{(\ell)})
					=	 \alpha_{\epsilon}(\nu^{(\ell)}).
	\]
	Due to Theorem \ref{thm:monotonicity}, part 1.,  we have
	$\widetilde{z} \in \Lambda_{\epsilon}(\nu^{(\ell)})$. Hence, $\widetilde{z}$ is also
	a rightmost point of $\Lambda_{\epsilon}(\nu^{(\ell)})$ in addition to $z^{(\ell)}$,
	a contradiction. This shows that $z^{(\ell)}$ is the unique rightmost point of 
	$\Lambda^{{\mathcal V}_k}_{\epsilon}(\nu^{(\ell)})$ as well.

	For the proof of differentiability of $\alpha^{{\mathcal V}_k}_{\epsilon}(\nu)$
	at $\nu = \nu^{(\ell)}$, we also need to show that
	$\sigma_{\min}( P^{V_k}(z^{(\ell)};\nu^{(\ell)}))$ is simple. For the sake of contradiction, suppose otherwise
	that this singular value is not simple. Then
	it follows from the first-order optimality conditions, in particular from 
	the analogue of (\ref{eq:KKT_poly_psa}) 
	where $\nu^{(\ell)}$ replaces $\vartheta$,
	that the Lagrange  multiplier $\mu(\nu^{(\ell)})$ is nonzero. Thus, by the complementarity
	condition $\mu(\nu^{\ell}) g(z^{(\ell)} ; \nu^{(\ell)}) = 0$ \cite[Theorem 12.1]{Nocedal2006},
	as $z^{(\ell)}$ is the maximizer of the optimization problem (\ref{eq:psa_opt})
	with $\nu = \nu^{(\ell)}$, we have $g(z^{(\ell)} ; \nu^{(\ell)}) = 0$
	or equivalently
	$
			\frac{\sigma_{\min}(P(z^{(\ell)};\nu^{(\ell)}))}{p_w( | z^{(\ell)} |)} = \epsilon	
	$.
	Combining this with
	\begin{equation*}
		\begin{split}
		\epsilon	\;	\geq	\;	
			\frac{\sigma_{\min}(P^{V_k}(z^{(\ell)};\nu^{(\ell)}))}{p_w( | z^{(\ell)} |)}
				\;	&	=	\;
			\frac{\| P(z^{(\ell)};\nu^{(\ell)}) V_k a\|}{p_w( | z^{(\ell)} |)}	
				\;	\geq		\;
			\frac{ \sigma_{\min}( P(z^{(\ell)};\nu^{(\ell)})) }{p_w( | z^{(\ell)} |)}	\,	,
		\end{split}
	\end{equation*}
	where $a$ is a unit norm right singular vector 
	corresponding to $\sigma_{\min}( P^{V_k}(z^{(\ell)};\nu^{(\ell)}))$,
	we deduce that
	\[
		\frac{ \sigma_{\min}( P(z^{(\ell)};\nu^{(\ell)})) }{p_w( | z^{(\ell)} |)}
			\;	=	\;
		\frac{\sigma_{\min}(P^{V_k}(z^{(\ell)};\nu^{(\ell)}))}{p_w( | z^{(\ell)} |)}
			\;	=	\;
		\epsilon	\,	.
	\]
	Since $\sigma_{\min}( P^{V_k}(z^{(\ell)};\nu^{(\ell)}))$ is not simple, there is 
	another unit right singular vector $\widetilde{a}$ 
	corresponding to $\sigma_{\min}( P^{V_k}(z^{(\ell)};\nu^{(\ell)}))$
	such that $a$ and $\widetilde{a}$ are orthogonal. But then
	\begin{equation*}
		\begin{split}
		\sigma_{\min}( P(z^{(\ell)};\nu^{(\ell)}))
			\;	&	=	\;
		\sigma_{\min}(P^{V_k}(z^{(\ell)};\nu^{(\ell)}))	\\
			\;	& 	=	\;
		\| P(z^{(\ell)};\nu^{(\ell)}) V_k a\|
			\;	 	=	\;
		\| P(z^{(\ell)};\nu^{(\ell)}) V_k \widetilde{a}\| \, ,
		\end{split}
	\end{equation*}
	yielding the contradiction that $V_k a$, $V_k \widetilde{a}$ are orthogonal,
	right singular vectors corresponding to $\sigma_{\min}( P(z^{(\ell)};\nu^{(\ell)}))$,
	that is $\sigma_{\min}( P(z^{(\ell)};\nu^{(\ell)}))$ is not simple.
	
	As $z^{(\ell)}$ is the unique rightmost point in 
	$\Lambda^{{\mathcal V}_k}_{\epsilon}(\nu^{(\ell)})$,
	the smallest singular value $\sigma_{\min}( P^{V_k}(z^{(\ell)};\nu^{(\ell)}))$ is simple, and
	the Hessian ${\mathcal L}^{{\mathcal V}_k}_{yy}
			(\nu^{(\ell)},\widehat{z}^{(\ell)}, \mu^{{\mathcal V}_k}(\nu^{(\ell)}))$
	is invertible by assumption, $\nu^{(\ell)}$ is a non-degenerate point for the
	restriction of $P(\cdot ; \cdot)$ to ${\mathcal V}_k$. Additionally,
	by the interpolation of the derivatives in (\ref{eq:psa_comp_HI}), we have 
	$g^{{\mathcal V}_k}_s(\widehat{z}^{(\ell)}; \nu^{(\ell)}) =
		g_s(\widehat{z}^{(\ell)}; \nu^{(\ell)}) \neq 0$. Consequently,
	$\alpha^{{\mathcal V}_k}_{\epsilon}(\nu)$ is differentiable 
	at $\nu = \nu^{(\ell)}$ by Theorem \ref{thm:poly_red_psa_der}.

	Finally, for the equality of the gradients, suppose that $u$, $v$ are consistent unit norm left, right
	singular vectors, respectively, corresponding to $\sigma := \sigma_{\min}( P(z^{(\ell)};\nu^{(\ell)}))$. 
	Due to the simplicity assumption on this singular value, we have $v \in {\mathcal V}_k$.
	As a consequence, there exists a unit norm vector $a$ such that 
	$v = V_k a$. Since $u$, $v$ are consistent 
	left, right singular vectors, respectively, we have
	$\, P(z^{(\ell)};\nu^{(\ell)}) v	\,	=	\,	\sigma u \,$
			and
	$\, u^\ast  P(z^{(\ell)};\nu^{(\ell)}) 	\,	=	\,	\sigma  v^\ast \,$,
	which can be rewritten in terms of $a$ as
	\[
		P^{V_k} (z^{(\ell)};\nu^{(\ell)}) a	\;	=	\;	\sigma u
			\quad	\text{and}	\quad
		u^\ast  P^{V_k}(z^{(\ell)};\nu^{(\ell)}) 	\;	=	\;	\sigma  a^\ast	\,	.
	\]
	Hence, $u$, $a$ are consistent unit left, right
	singular vectors, respectively, corresponding to $\sigma_{\min}( P^{V_k}(z^{(\ell)};\nu^{(\ell)}))$. 
	From the expressions (\ref{eq:psa_der}), (\ref{eq:lagrange_mult})
	for the derivatives of the pseudospectral abscissa function, and (\ref{eq:red_psa_der}), (\ref{eq:red_lagrange_mult})
	for the reduced counterparts, we obtain
	\begin{equation*}
		\begin{split}
		\frac{\partial \alpha_{\epsilon}}{\partial \nu_j} (\nu^{(\ell)}) \;
			&	 =	\;
			\frac{-\mre \left\{ u^\ast P_{\nu_j}( z^{(\ell)}; \nu^{(\ell)}) v \right\}}
				{\mre\{ u^\ast P'(z^{(\ell)} ; \nu^{(\ell)}) v \} - 
				\epsilon \{ \mre( z^{(\ell)} ) / | z^{(\ell)} | \}  p'_w( | z^{(\ell)} | )}		\\[.3em]
			&	=	\;
			\frac{-\mre \left\{ u^\ast P^{V_k}_{\nu_j}( z^{(\ell)}; \nu^{(\ell)}) a \right\}}
				{\mre\{ u^\ast P'(z^{(\ell)} ; \nu^{(\ell)}) V_k a \} - 
				\epsilon \{ \mre( z^{(\ell)} ) / | z^{(\ell)} | \}  p'_w( | z^{(\ell)} | )} \\[.7em]
					&	=	\;	
			\frac{\partial \alpha^{{\mathcal V}_k}_{\epsilon}}{\partial \nu_j} (\nu^{(\ell)})	
		\end{split}	
	\end{equation*}
	for $j = 1, \dots , \sd$, where $P_{\nu_j}( z; \nu)$, $P^{V_k}_{\nu_j}( z; \nu)$ denote the partial derivatives of
	$P$, $P^{V_k}$, respectively, with respect $\nu_j$ keeping $z$ and the entries of $\nu$ other than $\nu_j$ fixed.
\end{proof}

\subsection{Global convergence}\label{subsec:gl_conv}
The result below concerns the convergence of Algorithm~\ref{alg:sub_fw_min},
and is identical to \cite[Theorem 3.1]{AliM22}. This result follows from 
part 2. of Theorem \ref{thm:monotonicity} and part 1. of Theorem \ref{thm:hermite_interpolate}.
\begin{theorem}\label{thm:convergence}
Suppose that the iterates $\nu^{(\ell)}$ and $\nu^{(k)}$ by Algorithm \ref{alg:sub_fw_min}
are equal for some integers $\ell$, $k$ such that $1 \leq \ell < k$. Then $\nu^{(\ell)}$ is a
global minimizer of $\alpha_{\epsilon}(\nu)$ over all $\nu \in \underline{\Omega}$.
\end{theorem}
\begin{proof}
   The proof is the same as that  of \cite[Theorem 3.1]{AliM22}.
\end{proof}
\noindent
As we deal with finite dimensional $\sn \times \sn$ matrix polynomials, 
the restriction subspace ${\mathcal V}_k$ in Algorithm~\ref{alg:sub_fw_min} must 
stagnate after finitely many iterations, and Theorem~\ref{thm:convergence}  implies that when 
stagnation in the subspaces occurs, then the iterates stagnate at a global minimizer.

After deriving the theoretical basis of the subspace framework, in the next section 
we apply the results to damping optimization problems.


\section{Application to damping optimization}\label{sec:apply_damping_opt}
The classical damping optimization problem (see, e.g. \cite{Tom23}) 
concerns a second-order system of the form
\begin{equation}\label{eq:damping}
	M \ddot{x}(t) + C(\nu) \dot{x}(t) + K x(t) \; = \; f(t) \, ,
\end{equation}
where $M, K \in {\mathbb R}^{n\times n}$ are the symmetric positive definite 
mass, stiffness matrices, and $C(\nu) \in {\mathbb R}^{n\times n}$
is the symmetric positive semi-definite damping matrix dependent on
parameters, e.g. the viscosities of passive dampers. Typically, $C(\nu)$ can be partitioned as
\[
	 C(\nu) \; = \; C_{\text{int}} + C_{\text{ext}}(\nu)
\]
with $C_{\text{int}} \; = \; 2 \xi M^{1/2} \sqrt{M^{-1/2} K M^{-1/2}} M^{1/2}$
representing the internal damping for some constant $\xi > 0$, and 
$C_{\text{ext}}(\nu)$ representing the external damping that is determined
by the viscosities $\nu$, as well as the positions of the dampers.

There are several objectives considered in the literature to choose the
viscosities \cite{Nak02,Ves11}. One of the objectives that has been taken
into consideration is the minimization of the spectral abscissa \cite{FreL99,NakTT13}. 
In order to obtain robust stability, rather than minimizing the spectral abscissa,  
here we aim to choose the viscosities $\nu \in \underline{\Omega}$
minimizing the $\epsilon$-pseudospectral abscissa $\alpha_{\epsilon}(P(\cdot ; \nu))$
for a given $\epsilon > 0$, where $P(\lambda ; \nu) = \lambda^2 M + \lambda C(\nu) + K$, 
and $\underline{\Omega}$ is a prescribed box in the parameter space. 
In the last example we consider, in addition to a damper, that the matrix $K$ also depends
on a parameter, and we still minimize $\alpha_{\epsilon}(P(\cdot ; \nu))$ on a box $\underline{\Omega}$. 
In all of the examples in this section, by construction, 
the coefficients $M$, $C(\nu)$, $K$ are symmetric, and 
positive definite, positive semi-definite, positive definite, respectively.

All of the numerical experiments in this section were performed in MATLAB 2024b 
on a MacBook Air 
with Mac OS~15.4 operating system, Apple M3 CPU and 16GB RAM using our
implementations of Algorithm \ref{alg:sub_fw_min} and \ref{alg:sub_fw}, as well
as other ingredients such as the criss-cross algorithm (i.e. Algorithm \ref{alg:criss-cross}).
The reduced minimization problem in line \ref{siter_start_sf_min} of Algorithm \ref{alg:sub_fw_min}
is solved globally throughout by means of the package {\EIGOPT}. When using Algorithm~\ref{alg:sub_fw_min}, 
the number of initial interpolation points in line \ref{line:absk_sf_min0} is $\eta = 5$ and $\eta = 8$
when the problem is dependent on one and two parameters, respectively. The initial
interpolation points are chosen as equally spaced points on the interval $\underline{\Omega}$
if there is one parameter or on the line segment that connects 
the lower-left and upper-right corners of the box 
$\underline{\Omega}$ if there are two parameters. The prescribed termination tolerance
for Algorithm~\ref{alg:sub_fw_min} for the examples below is $10^{-8}$.

The MATLAB software with which we have carried out the numerical experiments
in this section is publicly available \cite{MehM24}. All of the results reported below 
can be reproduced using this software.

\begin{example}\label{sm:ex2}{\rm
This is a small-scale example arising from a mass-spring-damper system
consisting of four masses, where the damper is positioned on the second mass,
and we want to determine the optimal viscosity for this damper minimizing the 
$\epsilon$-pseudospectral abscissa objective. The system has the coefficients
\[
	M
		=
	\mathrm{diag}(     
	1,2,3,4)	\:	,	\;\;
	C(\nu)
		=
	C_{\mathrm{int}}	\,	+	\,	\nu e_2 e_2^T	\:	,	\;\;
	K
		=
	\mathrm{tridiag}(-5,10,-5) \, ,
\]
where $C_{\mathrm{int}} \, = \, 2 \xi M^{1/2} \sqrt{M^{-1/2} K M^{-1/2}} M^{1/2}$
with $\xi = 0.005$. 

We minimize the $\epsilon$-pseudospectral abscissa $\alpha_{\epsilon}(P(\cdot ; \nu))$ 
for $(w_m, w_c, w_k) = (1, 1, 1)$ and $\epsilon = 0.05$ over $\nu \in [0, 100]$
using {\EIGOPT} and Algorithm~\ref{alg:criss-cross} to compute the objective 
$\alpha_{\epsilon}(\cdot ; \nu)$ and its derivatives at various $\nu$. 

Computing the globally optimal value with an error less than $10^{-8}$
requires 414 $\epsilon$-pseudospectral abscissa evaluations, and takes   
about 0.7 seconds. The computed global minimizer $\nu_\ast = 4.6679$
(i.e. the optimal viscosity) is illustrated in the plots of $\alpha_{\epsilon}(P(\cdot ; \nu))$
as a function of the viscosity $\nu$ on the left in Figure \ref{fig:small_damping}.

The original undamped system is asymptotically stable 
with the spectral abscissa $\alpha(P(\, \cdot \, ;0)) = -0.0043$, 
but not robustly stable as $\alpha_{\epsilon}(P(\, \cdot \, ;0)) = 0.0619$ is positive indicating the existence
of nearby systems that are not asymptotically stable. On the other hand, the globally minimal 
value of the $\epsilon$-pseudospectral abscissa 
$\alpha_{\epsilon}(P(\, \cdot \, ;\nu_\ast)) = -0.0888$
is negative, so not only this damped system but also all nearby systems are asymptotically stable. 
Moreover, the spectral abscissa for the optimized system 
$\alpha(P(\, \cdot \, ;\nu_\ast)) = -0.1347$ 
is smaller than that for the undamped system by a magnitude of more than 
thirty times, indicating also
a better asymptotic behavior for the optimally damped system. The $\epsilon$-pseudospectrum
of the original undamped system and the optimally damped system are shown on the right
in Figure \ref{fig:small_damping}. While some of the components of the $\epsilon$-pseudospectrum
for the undamped system extend to the right-half of the complex plane, all components 
of the $\epsilon$-pseudospectrum of the optimally damped system lie in the left-half of
the complex plane. Indeed, optimization moves all of the components of the $\epsilon$-pseudospectrum
considerably to the left.

 \begin{figure}
 \vskip 4.5ex 
 	\begin{tabular}{cc}
		
			\hskip -2.5ex
			\includegraphics[width = .56\textwidth]{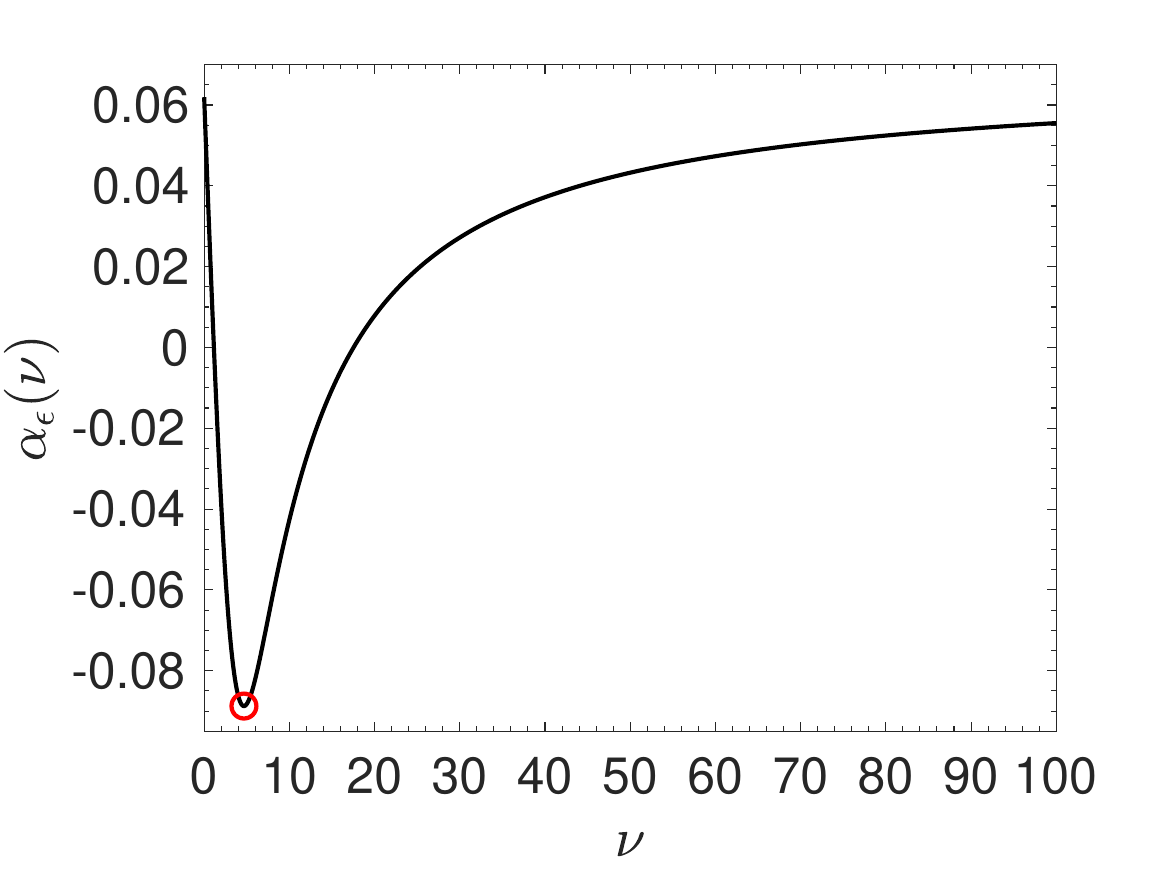}
				&		\\[-17em]
			\hskip -2.5ex
			\includegraphics[width = .56\textwidth]{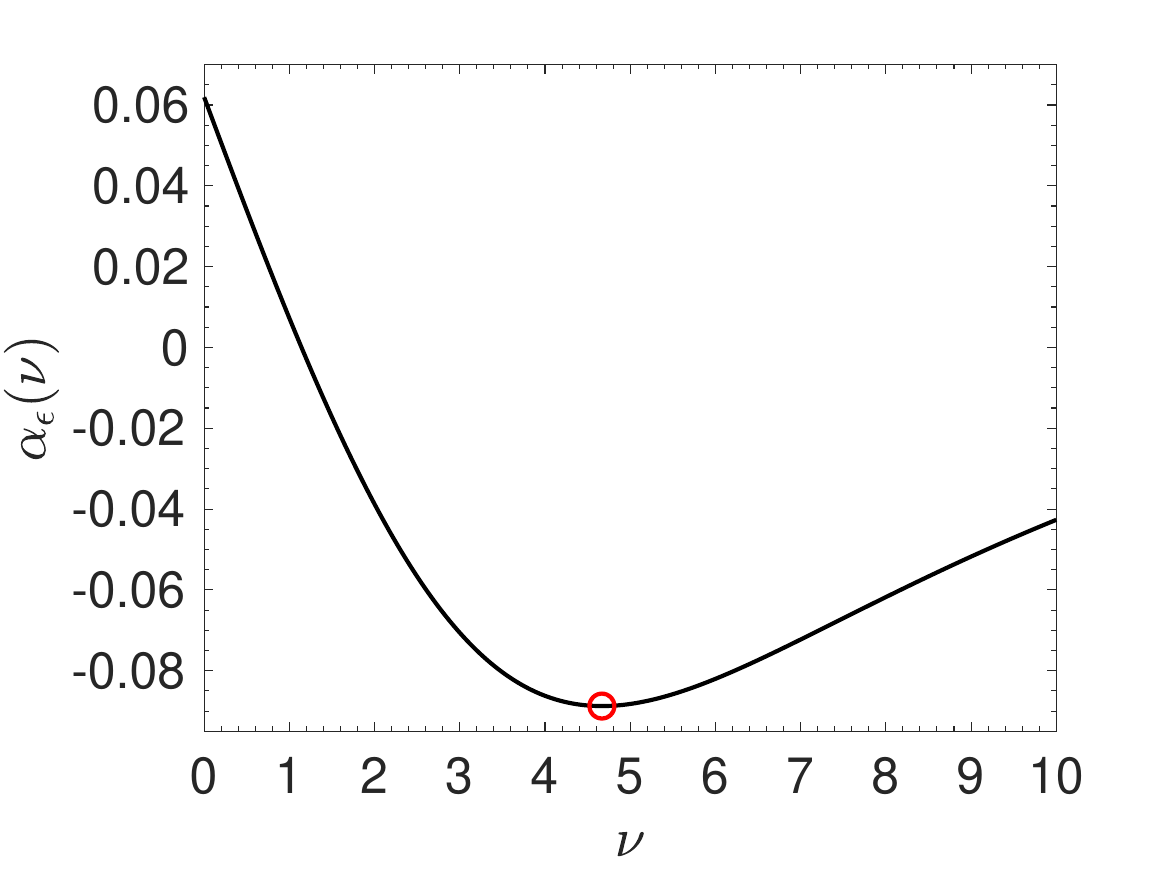} 
				&
			\hskip -3ex
			\includegraphics[width = .45\textwidth]{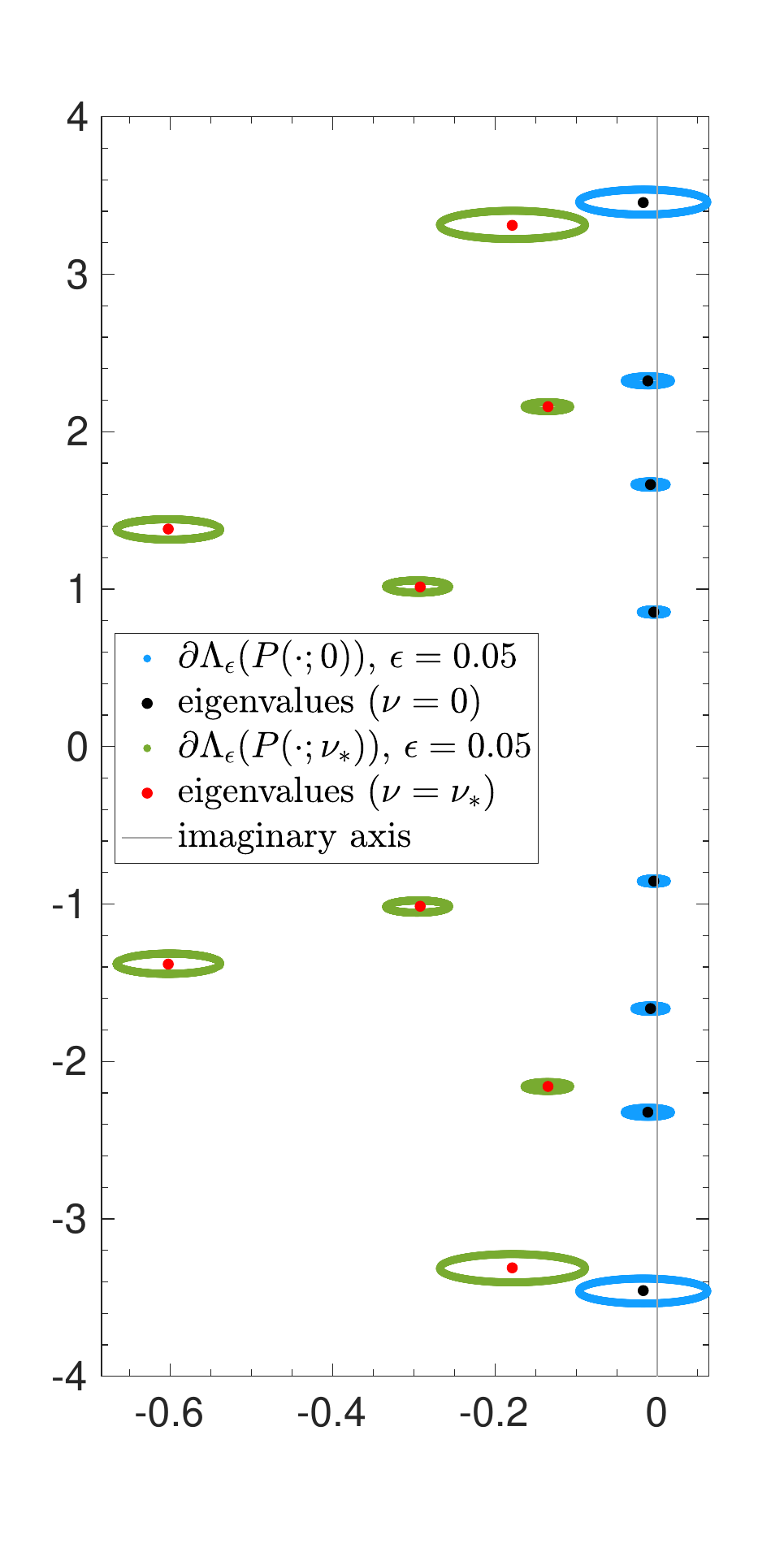}
	\end{tabular}
		\caption{ Small-scale damping optimization problem
		with $\epsilon = 0.05$ in Example \ref{sm:ex2}.
		 The left plots illustrate $\alpha_{\epsilon}(P(\cdot ; \nu))$
		as a function of the viscosity $\nu$, as well as its computed global minimizer 
		marked with red circles.
		 In the bottom plot,
		 the $\epsilon$-pseudospectral abscissa objective is zoomed near
		 the computed global minimizer. In the right plot, the boundary of the $\epsilon$-pseudospectrum 
		for the originally undamped system $\partial \Lambda_{\epsilon}(P(\cdot ; 0))$ (in blue)
		and that for the optimally damped system $\partial \Lambda_{\epsilon}(P(\cdot ; \nu_\ast))$ (in green)
		are shown along with their 
		eigenvalues. 
		}
		\label{fig:small_damping}
\end{figure}
}
\end{example}

\begin{example}\label{lar:ex1}{\rm
A larger example similar to Example \ref{sm:ex2} involves $20\times 20$
mass, damping, stiffness matrices given by
\[
	M
		=
	\mathrm{diag}(     
	1,\dots, 20)	\:	,	\;\;
	C(\nu)
		=
	C_{\mathrm{int}}	\,	+	\,	\nu e_2 e_2^T	\:	,	\;\;
	K
		=
	\mathrm{tridiag}(-25,50,-25) \, ,
\]
where as before $C_{\mathrm{int}} \, = \, 2\xi  M^{1/2} \sqrt{M^{-1/2} K M^{-1/2}} M^{1/2}$
with $\xi = 0.005$. As in the previous example, we minimize 
$\alpha_{\epsilon}(P(\cdot ; \nu))$ for $\epsilon = 0.05$ and $(w_m, w_c, w_k) = (1, 1, 1)$ 
over $\nu \in [0, 100]$, but now by applying the subspace framework Algorithm~\ref{alg:sub_fw_min}
with a prescribed error less than $10^{-8}$. After 
four iterations the method returns $\nu_\ast = 42.1076$ 
as the minimizer, as well as $\alpha_{\epsilon}(P(\cdot ; \nu_\ast)) = 0.0020$.
The accuracy of these computed optimal values
is illustrated in the plots on the left-hand side of Figure \ref{fig:larger_damping}.
It is apparent from the plots that, unlike in Example \ref{sm:ex2}, the $\epsilon$-pseudospectral 
abscissa objective is now nonsmooth at the global minimizer, but the method still converges 
to the global minimizer. 
The original undamped system is asymptotically stable 
with $\alpha(P(\cdot ; 0)) = -0.0011$, yet
small perturbations move its eigenvalues to the right-half plane, as 
$\alpha_{\epsilon}(P(\cdot ; 0)) = 0.1324$. The optimally damped system
has better asymptotic behavior with $\alpha(P(\cdot ; \nu_\ast)) = -0.0079$, and
has considerably smaller $\epsilon$-pseudospectral abscissa 
$\alpha_{\epsilon}(P(\cdot ; \nu_\ast)) = 0.0020$,
but there are still unstable systems in the $\epsilon$-neighborhood of the optimally damped system.
The right-hand plot in Figure \ref{fig:larger_damping} depicts the $\epsilon$-pseudospectrum
for the undamped and optimally damped system, where again it is evident that optimization
moves all of the components of the $\epsilon$-pseudospectrum to the left, but
not completely to the left-half of $\mathbb{C}$.

 \begin{figure}
\vskip 3ex
 	\begin{tabular}{cc}
			\hskip -2.5ex
			\includegraphics[width = .53\textwidth]{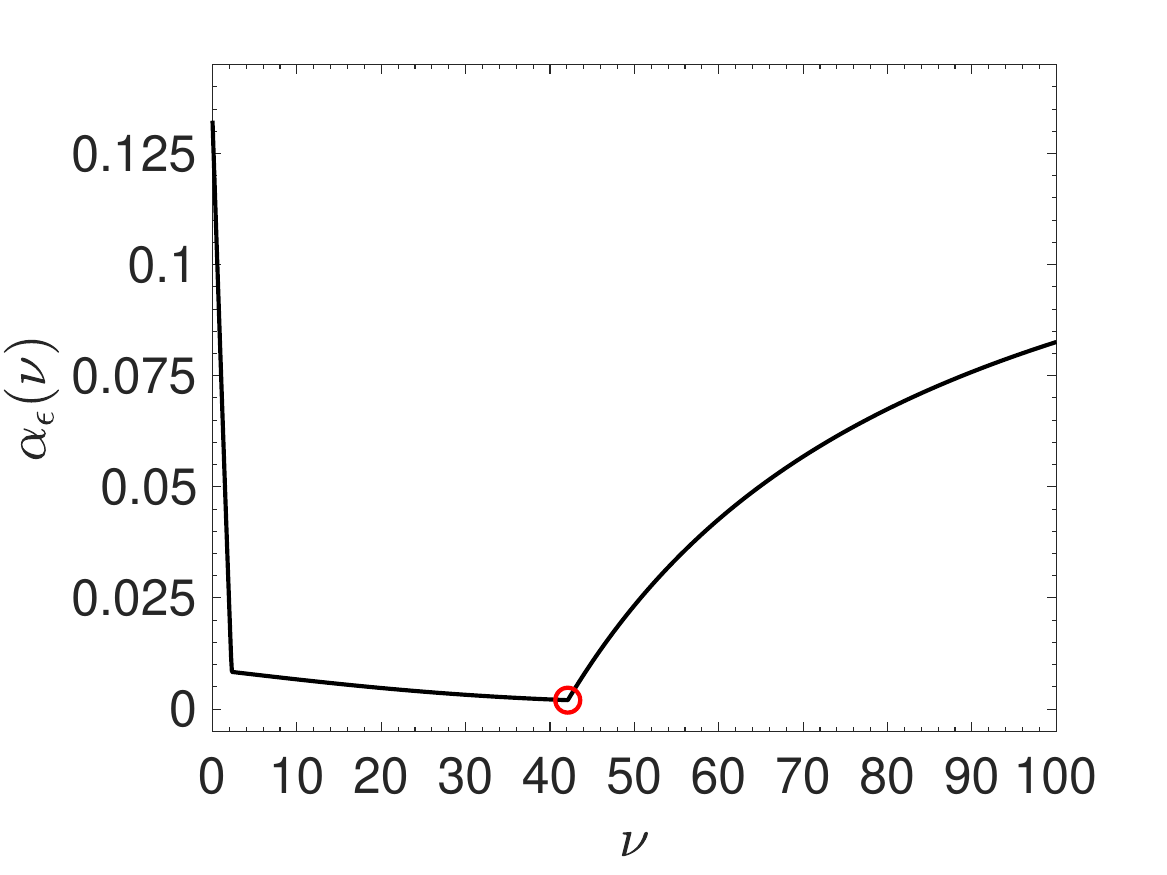}
				&		\\  [-15.5em]
			\hskip -2.5ex
			\includegraphics[width = .53\textwidth]{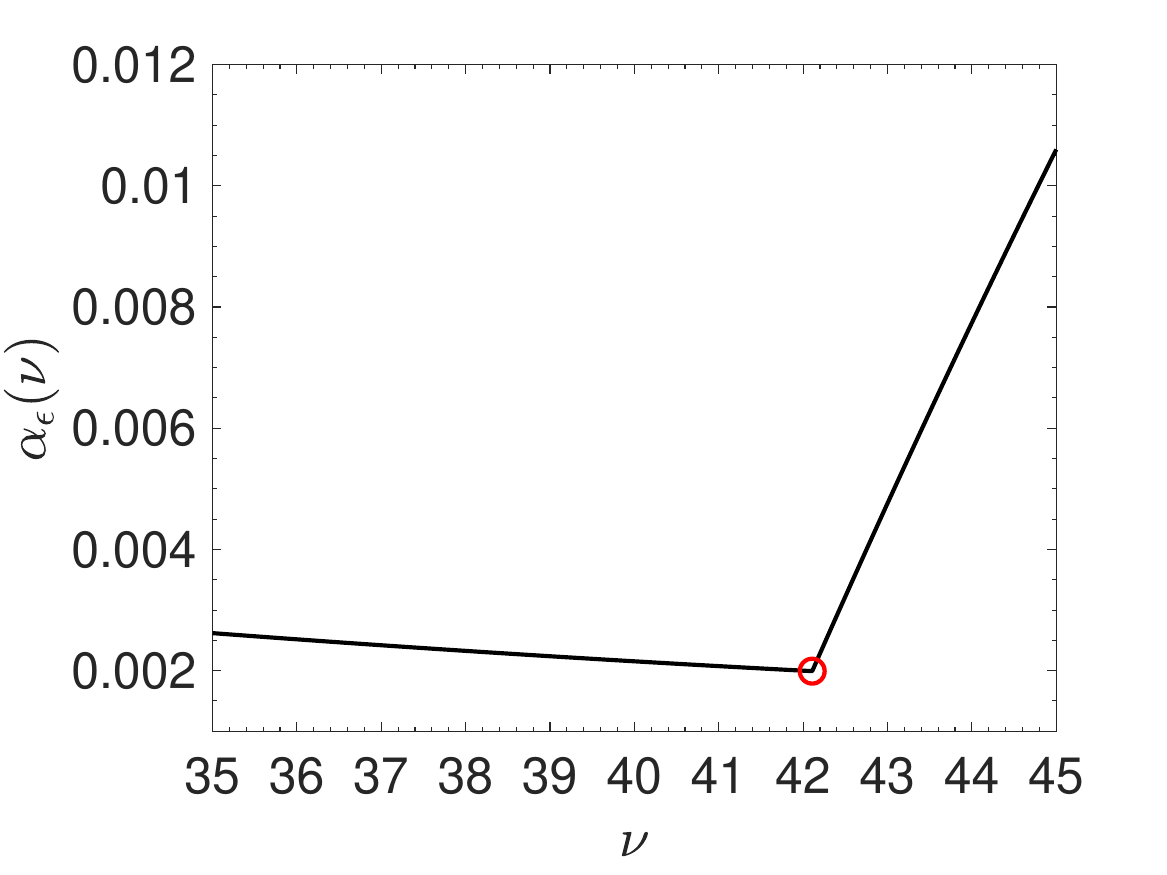} 
				&
			\hskip -3ex
			\includegraphics[width = .46\textwidth]{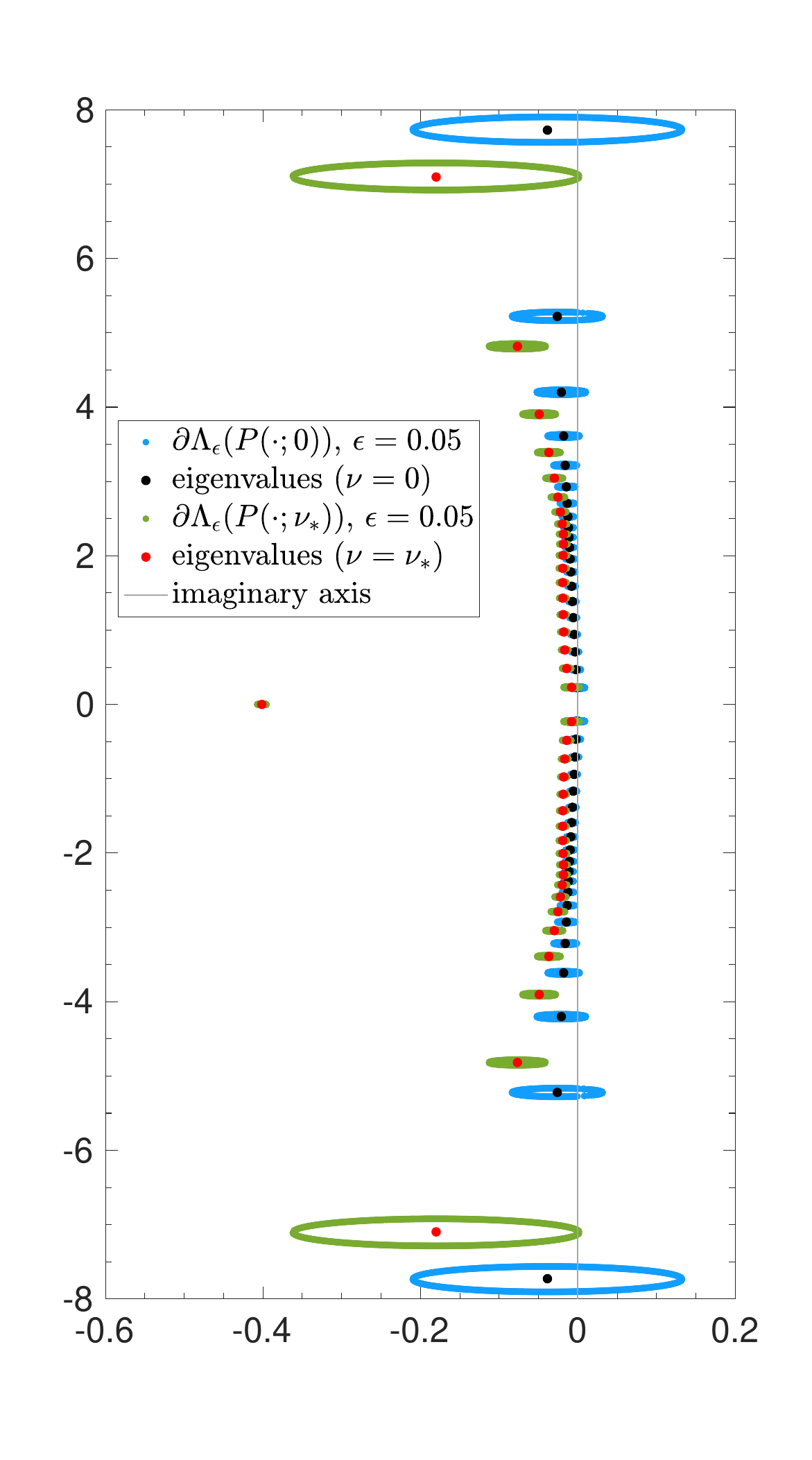}
	\end{tabular}
		\caption{ Results for
	$20\times 20$ damping problem in Example \ref{lar:ex1}.
		}
		\label{fig:larger_damping}
\end{figure}

The iterates $\nu^{(k)}$, $\alpha_{\epsilon}(\nu^{(k)})$ for $\, k\geq 1$ 
of Algorithm \ref{alg:sub_fw_min} listed in Table \ref{tab:sf_iterates} indicate a fast 
convergence. The total runtime of Algorithm \ref{alg:sub_fw_min} until termination is 
about 1.8 seconds. This is
much faster than a direct minimization (i.e. without employing subspace frameworks)
 of $\alpha_{\epsilon}(P(\cdot ; \nu))$ via {\EIGOPT} coupled with Algorithm~\ref{alg:criss-cross} to compute 
 the objective, which takes about 87.5 seconds. 
Even a direct minimization of $\alpha_{\epsilon}(P(\cdot ; \nu))$ 
via {\EIGOPT} coupled with the
subspace framework (i.e. Algorithm \ref{alg:sub_fw}) to compute the $\epsilon$-pseudospectral abscissa
takes about 25.4 seconds and is still substantially slow.
 The distribution of the runtime among the different parts  of Algorithm \ref{alg:sub_fw_min} is presented in the second, third, fourth columns of Table \ref{lar:ex1}.
 Most of the runtime of Algorithm~\ref{alg:sub_fw_min} is taken by the solutions of the
 reduced problems in line \ref{siter_start_sf_min}, and the computation of the rightmost point
 in $\Lambda_{\epsilon}(P(\cdot ; \nu))$  via Algorithm \ref{alg:sub_fw} 
 in lines \ref{line:absk_sf_min0} and \ref{line:absk_sf_min}. Among these two ingredients, 
solving the reduced problems contributes more to the total runtime of 
 Algorithm \ref{alg:sub_fw_min}.


\begin{table}
\begin{center}
\begin{tabular}{|c||c|c|c|}
        \hline
	       $k$ 	& 	$\nu^{(k)}$  	&  	$| \nu^{(k)} - \nu_\ast |$	&	$\alpha_{\epsilon}(\nu^{(k)})$	\\
	\hline
	\hline 
	1 &  \phantom{2}9.15901    &  $\, 3.29 \cdot 10^{1}$	&	-0.20179812	  \\ 
        2 & \underline{42.106}85   & $\;\;\; 7.66 \cdot 10^{-3}$	&	\phantom{-}\underline{0.0019}8921 \\ 
 	3 & \underline{42.10761}  & $ \: < 10^{-5}$	&	\phantom{-}\underline{0.00199163} \\ 
        \hline 
\end{tabular}
\end{center}
	\caption{ Iterates of Algorithm \ref{alg:sub_fw_min}
	to solve the damping optimization problem
	in Example \ref{lar:ex1}. The underlined digits indicate correct decimal digits.}
	\label{tab:sf_iterates}
\end{table}


%


\begin{table}
	\begin{center}
	\begin{tabular}{|c|c|c|c|c|c|c|}
        \hline
        $\:$ iter $\:$ 	&
	        $\;\,$total$\;\,$ 	 	& 		reduced  &  	\phantom{a} psa \phantom{a}  	&  	inner & direct (Alg.~\ref{alg:criss-cross})	& direct (Alg.~\ref{alg:sub_fw})		\\
	\hline
	\hline
	4	&	1.8 		& 		1.6  		&  	0.2  		& 	4, 6		&		87.5		&	25.4	   	\\ 	
        \hline 
	\end{tabular}
	\end{center}
	\caption{Number of iterations and runtimes in seconds to solve the damping problem
					in Example \ref{lar:ex1}. The total time (the column `total'),
					time taken by the reduced problems inside Algorithm \ref{alg:sub_fw_min} 
							(the column `reduced'),
							and time to compute $\alpha_{\epsilon}(P(\cdot ; \nu))$,
							corresponding rightmost point inside Algorithm \ref{alg:sub_fw_min} 
							(the column `psa')
							are also listed. 
		The column of `iter' corresponds to the number of (outer) iterations of Algorithm 	
		\ref{alg:sub_fw_min}, 
		while, in the column of `inner', the average number (rounded to nearest integer)
		and maximal number of 
		(inner) iterations performed by Algorithm~\ref{alg:sub_fw} during several calls
		throughout the minimization process via Algorithm \ref{alg:sub_fw_min} are given.
		The last two columns are times for direct minimization coupled
						with Algorithm~\ref{alg:criss-cross} or Algorithm~\ref{alg:sub_fw}
							to compute $\alpha_{\epsilon}(P(\cdot ; \nu))$.}
	\label{tab:sf_runtimes}
\end{table}

The inner subspace framework Algorithm \ref{alg:sub_fw} is called several times 
during this minimization via Algorithm \ref{alg:sub_fw_min}. 
The average number of iterations (rounded to nearest integer) and 
maximal number of iterations of Algorithm \ref{alg:sub_fw} 
during these calls are four and six, respectively.
We illustrate a typical convergence behavior of the inner 
subspace framework in Table \ref{tab:inner_iterates}, which lists the iterates of 
Algorithm \ref{alg:sub_fw} to compute $\alpha_{\epsilon}(P(\cdot; \nu_\ast))$.
A superlinear convergence of the iterates is apparent from the table.
Here, the inner subspace framework terminates after four iterations with a restriction
subspace of dimension four. This is a typical convergence behavior which seems
to be independent of the size of the matrix polynomial; the inner iteration terminates
with a subspace of dimension less than ten in all of our experiments. See in particular
the number of inner iterations reported below for the subsequent two examples.

\begin{table}
\begin{center}
\begin{tabular}{|c||c|}
        \hline
	       $k$ 	& 	$z^{(k)}$  	\\
	\hline
	\hline 
	1 &  \underline{0.00}041904 + \underline{0.230}22018$\mathbf{i}$	  \\ 
        2 &  \underline{0.001991}51 + \underline{0.23009}220$\mathbf{i}$ \\ 
 	3 &  \underline{0.00199163} + \underline{0.23009178}$\mathbf{i}$ \\ 
        \hline 
\end{tabular}
\end{center}
	\caption{The iterates of Algorithm \ref{alg:sub_fw} on the damping problem
	in Example \ref{lar:ex1} to compute $\alpha_{\epsilon}(P(\cdot ; \nu_\ast))$ at the
	global minimizer $\nu_\ast$.}
	\label{tab:inner_iterates}
\end{table}}


\end{example}

\begin{example}\label{lar:ex2}{\rm
To illustrate the behavior of the framework with increasing dimension $n$, this example 
concerns damping optimization problems with $n\times n$ mass, damping, 
stiffness matrices of the form
\begin{equation*}
\begin{split}
	&
	M
		=
	\mathrm{diag}(     
	1,2, \dots, n)	\:	,	\;\;
	C(\nu)
		=
	C_{\mathrm{int}}	\,	+	\,	\nu R^T R	\:	,	\;\;
	K
		=
	\mathrm{tridiag}(-400,800,-400) \, ,		\\[.7em]
	&
	\text{with}	\;\;
	R
		=
	\left[
		\begin{array}{rrrrrrrr}
		\; 0	&	\;\; 1	&	-1	&	\;\; 0	&	\;\; 0	&	\;\; 0	&	\dots		&	 0 	\\[.6em]	
		0	&	0	&	0	&	1	&	-1	&	0	&	\dots		&	 0 
		\end{array}
	\right]	\in {\mathbb R}^{2\times n}
\end{split}
\end{equation*}
and $C_{\mathrm{int}} \, = \, 2 \xi M^{1/2} \sqrt{M^{-1/2} K M^{-1/2}} M^{1/2}, \: \xi = 0.005 \:$
for $\: n = 80, 200, 400, 1200, 1400$.	
For every $n$, using Algorithm~\ref{alg:sub_fw_min}, we minimize $\alpha_{\epsilon}(\nu)$ over 
$\nu \in [0, 250]$ for 
$\epsilon = 0.03$ and two different choices for the weights, namely 
$(w_m, w_c , w_k) = (1, 1, 1)$ and $(w_m, w_c , w_k) = (0.7, 1, 0)$.
Note that for the latter choice, perturbations of the stiffness matrix are not allowed. 
Physically, the parameter-dependent matrix polynomial again corresponds to a mass-spring-damper system
consisting of $n$ masses connected with springs such that the second and third mass, as well as
the fourth and fifth mass are tied through dampers. Moreover, both of the two dampers are required
to have the same viscosity.

 The computed results of Algorithm~\ref{alg:sub_fw_min}  are reported in Table \ref{tab:sf_large}.
 In each case, the $\epsilon$-pseudospectral abscissa of the optimally damped system is considerably smaller than 
 that of the original undamped system. In some cases, even the $\epsilon$-pseudospectral abscissa becomes negative
 after minimization, indicating not only the optimal system but also nearby systems are made asymptotically stable
 through optimization. The number of subspace iterations of and runtimes of Algorithm~\ref{alg:sub_fw_min}  
 are also listed in Table \ref{tab:sf_large}. An important 
 observation is that the number of subspace iterations of Algorithm~\ref{alg:sub_fw_min}
 until termination remains small, indeed never exceeds three, in all cases. 
Similarly, according to the figures in the last column of Table \ref{tab:sf_large},
 the average and maximal number of inner iterations of Algorithm \ref{alg:sub_fw}
 throughout an execution of Algorithm~\ref{alg:sub_fw_min} never exceeds 
 five and eight, respectively. 
 The dimension of the restriction subspace for the outer 
 subspace framework (Algorithm~\ref{alg:sub_fw_min})
 is bounded by $\eta + \text{iter} - 1$, which, recalling that $\eta = 5$ 
 initial interpolation points are used, is less than eight in all cases. Also,
 the dimension of the restriction subspace for the inner
 subspace framework (Algorithm~\ref{alg:sub_fw}) never exceeds eight
 according to the maximal number of inner iterations
 listed in the last column of Table \ref{tab:sf_large}.
 
 As we have observed before, also looking at the times reported in the columns 
 of `total', `reduced', `psa' of Table \ref{tab:sf_large}, the total runtime 
 is more or less determined by (i) the time to solve the reduced problems 
 and (ii) the time to locate the rightmost points  
 in the $\epsilon$-pseudospectrum of the full problems. In these test examples for smaller
 matrix polynomials,
 the former contributes significantly more to the computation time. However, 
 the contribution of the latter becomes substantial as the sizes of the matrices increase.
 For larger matrix polynomials with $n > 1000$, according to Table \ref{tab:sf_large}, 
 computing the rightmost points in the $\epsilon$-pseudospectrum of the full 
 problems via Algorithm~\ref{alg:sub_fw} dominates the runtime.
 In such a call to Algorithm~\ref{alg:sub_fw} for the computation of 
 $\alpha_{\epsilon}(P(\cdot ; \nu))$ $\exists \nu \in \underline{\Omega} \,$
 and the parameter-dependent matrix polynomial 
 $P$ with $n = 1200$ or $n = 1400$, at least about $\% 95$ of the 
 total time is taken by the large-scale eigenvalue and singular value computations in lines  
 \ref{alg:init_points2_sf}, \ref{line:init_sub}, \ref{line:expand} of Algorithm~\ref{alg:sub_fw}.

Both the outer and the inner subspace frameworks 
are key ingredients for an efficient minimization of $\alpha_\epsilon(P(\cdot ; \nu))$. 
To illustrate this point, for the example with $n = 80$, 
we report the total runtimes in Table \ref{tab:only_one_sf}
when both of the frameworks are employed, as opposed to when only one of the two subspace 
frameworks is adopted, while the other one is abandoned. The column~`Alg.~\ref{alg:sub_fw_min} (with Alg.~\ref{alg:sub_fw})' in the table is the approach proposed
here when both of the subspace frameworks are employed, while the two columns to the
right are when the inner framework Algorithm \ref{alg:sub_fw} is omitted, and when 
the outer framework Algorithm \ref{alg:sub_fw_min} is omitted. 
Not using either one of the frameworks increases the runtime considerably,
but the effect of excluding the outer framework (i.e. Algorithm \ref{alg:sub_fw_min}) 
on the runtime is more pronounced.


\begin{table}
	\begin{center}
	\begin{tabular}{|r||c|c|c|c|c|c|c|c|}
        \hline
        		$n$, $\: w$	&
        		$\nu_\ast$		&	$\alpha_{\epsilon}(\nu_\ast)$  	&	$\alpha_{\epsilon}(0)$  	&
				iter		&	total		&	reduced  &  	psa		&	inner		\\
	\hline 
	\hline	
	$80$, $\mathbf{1}$		&	122.48	&	\phantom{-}0.00223	&	0.25226	&	3	&	12.3		&	11.7		&	0.5	&	3, 6		\\	
	\hline
$80$, $\mathbf{u}$		&	226.67	&  -0.00037		&	0.13030	&	3	&	18.4		&	17.9	 &	0.5	&	3, 3			\\ 
	\hline
		$200$, $\mathbf{1}$	&	123.50	&	\phantom{-}0.00438		&	0.25226		&	3	&	73.2	&	67.9	& 	5.0	&	4, 7		\\	
	\hline
		$200$, $\mathbf{u}$	&	227.20	&	-0.00002		&	0.13030		&	3		&	111.4	&	107.1		&		4.2	&	3, 3		\\   
	\hline
		$400$, $\mathbf{1}$		&	124.56	&	\phantom{-}0.00661	 &  	0.25226	&	3	&	260.1		&	227.1	&	33	&	5, 8			\\	
	\hline
		$400$, $\mathbf{u}$		&	227.25	&	\phantom{-}0.00001		&	0.13030	&	2	&	142.7	&	120.1	&	22.4		&	3, 3		\\ 
	\hline
		$1200$, $\mathbf{1}$		&	125.48	&	\phantom{-}0.00850	 &  	0.25226	&	2	&	820.3		&	98.3	&	692.8	&	5, 7		\\	
	\hline
		$1200$, $\mathbf{u}$		&	227.25	&	\phantom{-}0.00001	 &  	0.13030	&	2	&	774.6		&	134.9	&	618.5	&	3, 3	\\	
	\hline
		$1400$, $\mathbf{1}$		&	125.48	&	\phantom{-}0.00850	 &  	0.25226	&	2	&	1209.1		&	98.2	&	1091.1	&	5, 6		\\	
	\hline
		$1400$, $\mathbf{u}$		&	227.25	&	\phantom{-}0.00001	 &  	0.13030	&	2	&	1136.8		&	137.4	&	989.6	&	3, 3	\\	
	\hline
	\end{tabular}
	\end{center}
	\caption{ Application of Algorithm~\ref{alg:sub_fw_min} to solve the 
	damping problems of size $n = 80, 200, 400, 1200, 1400$ in Example \ref{lar:ex2},
	and with the weights $w = (w_m, w_c, w_k)$ equal to $\mathbf{1} := (1, 1, 1)$ or
	$\mathbf{u} := (0.7, 1, 0)$. The second, and
	third columns list the computed global minimizer, and
	the corresponding minimal value of $\alpha_{\epsilon}(\nu)$.
	The last five columns are as in Table \ref{tab:sf_runtimes}.}
	\label{tab:sf_large}
\end{table}

\begin{table}
	\begin{center}
	\begin{tabular}{|c||c|c|c|}
		\hline
		$w$	&
		Alg.~\ref{alg:sub_fw_min} (with Alg.~\ref{alg:sub_fw})	&	
		Alg.~\ref{alg:sub_fw_min} (with Alg.~\ref{alg:criss-cross})	&
		direct (with Alg.~\ref{alg:sub_fw})	\\
		\hline
		\hline
		$\mathbf{1}$	&	12.3		&		41.4		&	258.4		\\
		\hline
		$\mathbf{u}$	&	18.4		&		45.9		&	251.5		\\
		\hline
	\end{tabular}
	\end{center}
	\caption{Runtimes in seconds to minimize $\alpha_{\epsilon}(P(\cdot ; \nu))$ for
	the damping problem with $n = 80$ in Example \ref{lar:ex2}
	and the weights $w = (w_m, w_c, w_k)$ equal to either ${\mathbf 1} = (1, 1, 1)$
	or ${\mathbf u} = (0.7, 1, 0)$.}
	\label{tab:only_one_sf}
\end{table}

}
\end{example}

\begin{example}\label{ex:mul_par}{\rm
We finally discuss two damping optimization problems that depend on two parameters.
Both of the problems are variations of Example \ref{lar:ex1}. In the first problem,
the mass, damping, stiffness matrices are $20\times 20$ and given by
\[
	M
		=
	\mathrm{diag}(     
	1,\dots, 20)		,	\;
	C(\nu_1 , \nu_2)
		=
	C_{\mathrm{int}}		+		\nu_1 e_2 e_2^T		+		\nu_2 e_{19} e_{19}^T
			,	\;
	K
		=
	\mathrm{tridiag}(-25,50,-25),
\] 
where $C_{\mathrm{int}}$ is as in Example \ref{lar:ex1}. Thus, there are now two
dampers on the second and nineteenth mass instead of only one damper
on the second mass in Example \ref{lar:ex1}. The optimization parameters $\nu_1$ and $\nu_2$,
the viscosities of the dampers, are constrained to lie in the intervals $[0, 50]$ and $[0,100]$,
respectively, i.e. $\nu = (\nu_1 , \nu_2) \: \in \: \underline{\Omega} = [0,50]\times [0,100]$. 
An application of Algorithm~\ref{alg:sub_fw_min} to minimize $\alpha_{\epsilon}(\nu)$
for $\epsilon = 0.05$ and $(w_m , w_c , w_k) = (1, 1, 1)$ 
over $\nu \in \underline{\Omega} \,$ yields $\nu_\ast = (27.5958 , 62.1559)$ 
as the optimal values of the parameters. 
The accuracy of this computed minimizer is confirmed by the left-hand plot
in Figure \ref{fig:2parameter}. The quantities related to this application
of Algorithm~\ref{alg:sub_fw_min} are summarized in Table \ref{tab:2param_1}.
The computed globally minimal value of 
the $\epsilon$-pseudospectral abscissa is $\alpha_{\epsilon}(\nu_\ast) = -0.0186$.
Recall from the discussions following Example \ref{lar:ex1} that the $\epsilon$-pseudospectral
abscissa of the original undamped system is $0.1324$, while with only one damper on the
second mass the smallest $\epsilon$-pseudospectral abscissa possible is $0.0020$.
Using a second damper moves all of the components of the $\epsilon$-pseudospectral
abscissa completely to the open left-half of the complex plane, i.e. all
of the systems at a distance of $\epsilon = 0.05$ or closer are asymptotically stable for the
optimal choices for the two dampers. This robust asymptotic stability is not attainable with only 
one damper on the second mass as illustrated in Example \ref{lar:ex1}.

\begin{figure}
 \centering
	\begin{tabular}{cc}
		\hskip -2.2ex
			\includegraphics[width = .5\textwidth]{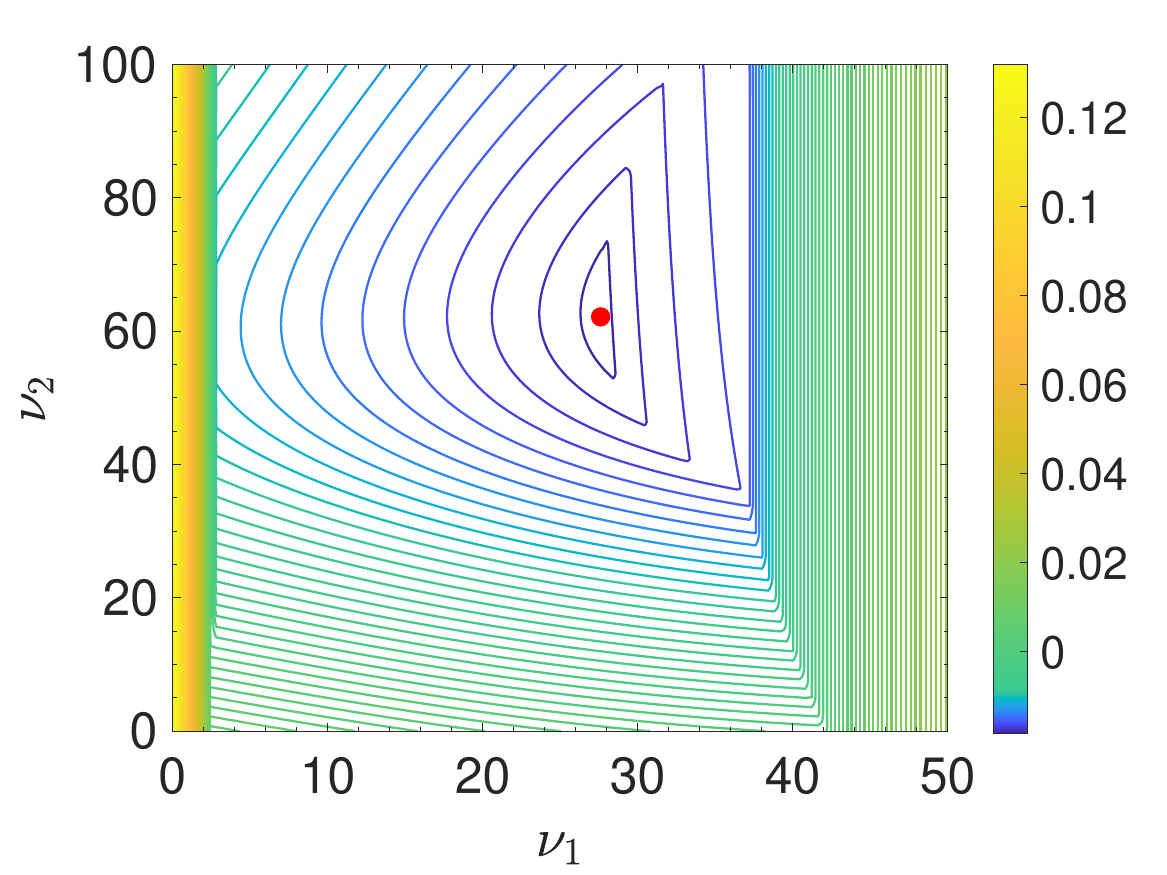} & 
			\hskip .7ex
			\includegraphics[width = .5\textwidth]{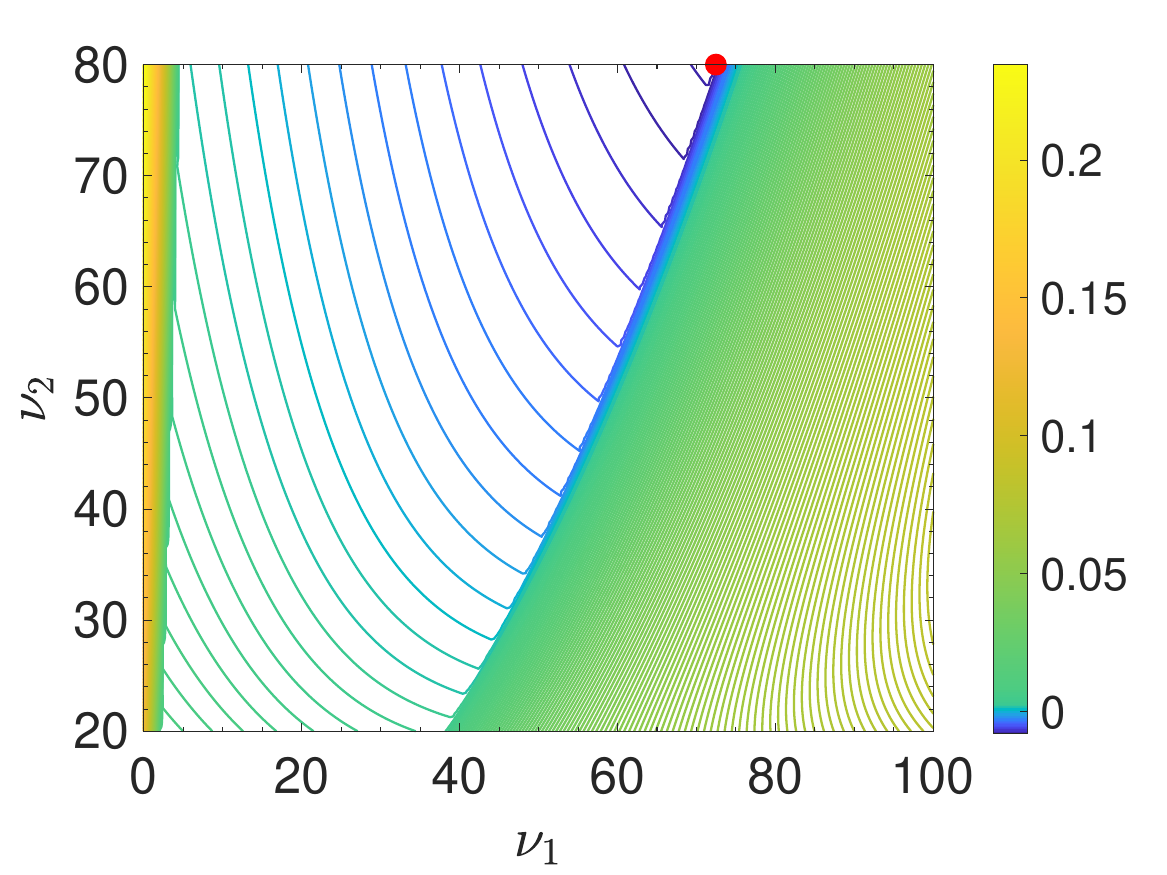}  	\\[-1em]
	\end{tabular}
		\caption{Contour diagrams of the objective function
		$\nu = (\nu_1 , \nu_2) \mapsto \alpha_{\epsilon}(\nu)$ for the
		two damping optimization problems in Example~\ref{ex:mul_par}
		depending on two parameters.
		In each plot, the red dot marks the computed global minimizer.
		Left and right plots correspond to first and second damping problems, respectively,
		considered in Example~\ref{ex:mul_par}.
		}
		\label{fig:2parameter}
\end{figure}

\begin{table}
	\begin{center}
	\begin{tabular}{|c|c|c|c|c|c|c|}
        \hline
        		$\nu_\ast$		&	$\alpha_{\epsilon}(\nu_\ast)$  	&	
			$\:$ iter $\:$		&	$\:$ total $\:$		&	reduced  &  	$\:$ psa $\:$	&	inner	\\
	\hline
	\hline
		(27.60, 62.16)		&	 -0.01865		&	5	&	54.7		&	54.3	&	0.4	&	4, 6		\\
	\hline	
	\end{tabular}
	\end{center}
		\caption{ Damping optimization problem with two dampers
		in Example \ref{ex:mul_par}.
		The computed global minimizers, globally smallest value of $\alpha_{\epsilon}(\nu)$ 
		by Algorithm~\ref{alg:sub_fw_min}, as well as the number of outer and inner
		subspace iterations, runtimes of Algorithm~\ref{alg:sub_fw_min} in seconds
		are listed as in Table \ref{tab:sf_large}.}
	\label{tab:2param_1}
\end{table}

In the second problem, the mass, damping, stiffness matrices are again $20\times 20$
but now given by
\[
	M
		=
	\mathrm{diag}(     
	1,\dots, 20)	\:	,	\;\,
	C(\nu_1 , \nu_2)
		=
	C_{\mathrm{int}}(\nu_2)	\,	+	\,	\nu_1 e_2 e_2^T	
		\:	,	\;\,
	K(\nu_2)
		=
	\nu_2
	\mathrm{tridiag}(-1,2,-1)
\]
with 
$C_{\mathrm{int}}(\nu_2) \, = \,  \, 2 \xi M^{1/2} \sqrt{M^{-1/2} K(\nu_2) M^{-1/2}} M^{1/2}, \: \xi = 0.005 \:$.
Thus, now there is only one damper on the second mass, whose viscosity $\nu_1$ is an optimization parameter.
Additionally, the spring constant $\nu_2$ affecting the stiffness matrix
is the second optimization parameter. We minimize
$\alpha_{\epsilon}(\nu)$ for $\epsilon = 0.05$ and $(w_m , w_c , w_k) = (1, 1, 1)$ over 
$\nu = (\nu_1, \nu_2) \: \in \: [0,100] \times [20,80]$ using Algorithm~\ref{alg:sub_fw_min}.
The computed global minimizer is $\nu_\ast = (72.4622, 80)$. The right-hand plot
in Figure \ref{fig:2parameter} illustrates the accuracy of this computed minimizer.
Additionally, quantities related to the application of Algorithm~\ref{alg:sub_fw_min} are provided in
Table \ref{tab:2param_2}. The computed globally smallest value of the $\epsilon$-pseudospectral
abscissa is now $\alpha_{\epsilon}(\nu_\ast) = -0.0081$, and again 
indicates uniform asymptotic stability
in an $\epsilon$ neighborhood of the optimal system,
unlike the minimal value $0.0020$ when minimization is performed only with respect to the
damping parameter on the second mass.


\begin{table}
	\begin{center}
	\begin{tabular}{|c|c|c|c|c|c|c|}
        \hline
        		$\nu_\ast$		&	$\alpha_{\epsilon}(\nu_\ast)$  	&	
				$\:$ iter $\:$		&	$\:$ total $\:$		&	reduced  &  	$\:$ psa $\:$		&	inner		\\
	\hline
	\hline
		(72.46, 80)		&	 -0.00805		&	3	&	15.1		&	14.8	 &	0.3	&	4, 6			\\
	\hline	
	\end{tabular}
	\end{center}
		\caption{ This table is similar to Table \ref{tab:2param_1} but for the application
				of Algorithm~\ref{alg:sub_fw_min}
				to the second damping problem in Example \ref{ex:mul_par}
				with the damping viscosity and the spring constant as parameters. }
	\label{tab:2param_2}
\end{table}


}
\end{example}

\section{Minimizing the pseudospectral abscissa of linearizations}\label{sec:psa_linearizations}  
In this section we discuss the minimization of the pseudospectral abscissa for the case 
when the quadratic eigenvalue problem is first linearized. 
Let us assume for this that the leading coefficient $M$ of the quadratic polynomial 
$P(\lambda) = \lambda^2 M + \lambda C + K$ is invertible. This is often the case
in real applications such as the damping optimization and
brake squealing problems; see, e.g. \cite{GraMQ2015,Tom23}. Then $P(\lambda) v = 0$ holds 
if and only if
\begin{equation}\label{eq:linearize}
	{\mathcal A}
	\left[
		\begin{array}{c}
			\lambda v	\\
			v
		\end{array}
	\right]
			=	
	\lambda
	\left [
		\begin{array}{c}
			\lambda v	\\
			v
		\end{array}
	\right]\:	,
	\;\:
	\text{where}	\;
	{\mathcal A}
		:=
	\left[
		\begin{array}{cc}
			-M^{-1}	&	0	\\
				0	&	I
		\end{array}
	\right]
	\left[
		\begin{array}{cc}
			C	&	K	\\
			I	&	0
		\end{array}
	\right]\in {\mathbb C}^{2\sn \times 2\sn}.
\end{equation}
This corresponds to the companion linearization of the polynomial $P(\lambda)$.
There are many alternative
linearizations considered in the literature (see, e.g. \cite{Higham2006,Mackey2006b,Mackey2006}).
We adopt the one above, but the ideas here are applicable to other linearizations as well. 

Clearly, the polynomial $P$ and the matrix ${\mathcal A}$ have the same set of
eigenvalues, and the relation between their eigenvectors is apparent from (\ref{eq:linearize}).
Hence, it seems plausible to consider the $\epsilon$-pseudospectrum of the matrix ${\mathcal A}$
given by \cite{Trefethen2005}
\[
	\Lambda^{\stan}_{\epsilon}({\mathcal A})	\:		:=	\:\;
	\bigcup_{\Delta \in {\mathbb C}^{2\sn \times 2\sn} \, , \; \| \Delta \|_2 \leq \epsilon} \, \Lambda({\mathcal A} + \Delta)	\\
				\quad	=	\quad
	\{ z \in {\mathbb C} \; | \;	\sigma_{\min}({\mathcal A} - zI)	\leq	\epsilon	\}	\:				
\]
with $\Lambda({\mathcal A} + \Delta)$ denoting the set of eigenvalue of the matrix 
${\mathcal A} + \Delta$, and the corresponding $\epsilon$-pseudospectral abscissa
\begin{equation*}
\begin{split}
	\alpha^{\stan}_{\epsilon}({\mathcal A})
		\:	&	:=	\:
	\max\{	\text{Re}(z)	\;	|	\;	z \in \Lambda^{\mathsf{st}}_{\epsilon}({\mathcal A})	\}		\\[.1em]
			&	\phantom{:}=	\;
	\max\{ 	\text{Re}(z)	\;	|	\;	z\in {\mathbb C}	\text{ s.t. }	
											\sigma_{\min}({\mathcal A} - zI)	\leq	\epsilon	\}	\:	.
\end{split}		
\end{equation*}

The next theorem relates $\Lambda_{\epsilon}(P)$ with 
$\Lambda^{\stan}_{\epsilon}({\mathcal A})$, as well as
$\alpha_{\epsilon}(P)$ with $\alpha^{\stan}_{\epsilon}({\mathcal A})$.
\begin{theorem}\label{thm:psa_poly_vs_stan}
Consider the quadratic matrix polynomial $P(\lambda)$ with $M$  invertible,
and let $\Lambda_{\epsilon}(P)$, $\alpha_{\epsilon}(P)$ be the $\epsilon$-pseudospectrum,
$\epsilon$-pseudospectral abscissa of $P$ with the weights $(w_m, w_c, w_k) = (0, 1, 1)$.
The following assertions hold for every $\epsilon > 0$: \\
1. $\Lambda^{\stan}_{\underline{\epsilon}}({\mathcal A})	\supseteq		\Lambda_{\epsilon}(P)$;$\:$
2. $\alpha^{\stan}_{\underline{\epsilon}}({\mathcal A})		\geq		\alpha_{\epsilon}(P)$,
$\:$where $\underline{\epsilon} := \epsilon\| M^{-1} \|_2 $.
\end{theorem}
\begin{proof} 
1. Suppose that  $z  \in \Lambda_{\epsilon}(P)$. Then there must exist $\Delta C, \Delta K$
	such that  $\| [ \Delta C \; \Delta K ]  \|_2 \leq \epsilon$ and
	$z$ is an eigenvalue $P(\lambda) + \Delta P(\lambda)$ for
	$\Delta P(\lambda) = \lambda \Delta C + \Delta K$.
	But then $z$ is also an eigenvalue of the matrix
\begin{equation*}
	\begin{split}
		&
		\left[
		\begin{array}{cc}
			-M^{-1}	&	0	\\
				0	&	I
		\end{array}
		\right]
		\left[
		\begin{array}{cc}
			C + \Delta C	&	K + \Delta K	\\
			I	&	0
		\end{array}
		\right]
					=	{\mathcal A}
			+	\Delta		\\
		&
		\quad\quad
		\text{for}	\quad
		\Delta :=
		\left[
		\begin{array}{cc}
			-M^{-1}	&	0	\\
				0	&	I
		\end{array}
		\right]
		\left[
		\begin{array}{cc}
			\Delta C	&	 \Delta K	\\
			0	&	0
		\end{array}
		\right]	\:	,
	\end{split}
	\end{equation*}
	where 
$
		\| \Delta \|_2	\;	\leq	\;	
		\| M^{-1} \|_2	\| [ \Delta C \; \Delta K ]  \|_2 	\;	\leq	\; 
		\epsilon\| M^{-1} \|_2 	\:	.
$
This shows that $z \in \Lambda^{\stan}_{\underline{\epsilon}}({\mathcal A})$
for $\underline{\epsilon} := \epsilon\| M^{-1} \|_2 $. Hence, the inclusion
$\Lambda^{\stan}_{\underline{\epsilon}}({\mathcal A})	\supseteq		\Lambda_{\epsilon}(P)$ holds.
	
\noindent	
2. This is immediate from the inclusion 
	$\Lambda^{\stan}_{\underline{\epsilon}}({\mathcal A})	\supseteq		\Lambda_{\epsilon}(P)$.
\end{proof}
Returning to the parameter-dependent matrix polynomial 
$P(\lambda ; \nu) = \lambda^2 M(\nu) + \lambda C(\nu) + K(\nu)$ under the assumption 
that $M(\nu)$ remains invertible for all $\nu \in \underline{\Omega}$, 
letting
\[
	{\mathcal A}(\nu)
		:=
	\left[
		\begin{array}{cc}
			-M(\nu)^{-1}	&	0	\\
				0	&	I
		\end{array}
	\right]
	\left[
		\begin{array}{cc}
			C(\nu)	&	K(\nu)	\\
			I	&	0
		\end{array}
	\right] \: ,
\]
one may alternatively minimize $\alpha^{\stan}_{\epsilon}({\mathcal A}(\nu))$ over $\nu \in \underline{\Omega}$
based on the stability considerations. This involves the minimization of the $\epsilon$-pseudospectral
abscissa of the parameter-dependent matrix ${\mathcal A}(\nu)$, so falls into the setting of 
\cite{AliM22, AliM22b}.

If the weights in the definitions of $\Lambda_{\epsilon}(P(\cdot;\nu))$
and $\alpha_{\epsilon}(P(\cdot;\nu))$ are chosen as  $(w_m, w_c, w_k) = (0, 1, 1)$,
and $M(\nu) = M$ is independent of the parameters $\nu$,
then, according to Theorem \ref{thm:psa_poly_vs_stan}, we have
\[
	\min_{\nu \in \underline{\Omega}} \:  \alpha_{\epsilon}(P(\cdot;\nu))
		\;\;	\leq		\;\;
	\min_{\nu \in \underline{\Omega}} \:	\alpha^{\stan}_{\underline{\epsilon}}({\mathcal A}(\nu))
\]
with $\underline{\epsilon} =\epsilon \| M^{-1} \|_2 $. However, the gap between the optimal
values of the two minimization problems above can be significant, and the corresponding
optimal parameter values can be considerably different; see in particular Example \ref{ex_linearization}
given next.

\begin{example}\label{ex_linearization}{\rm
For a comparison of $\alpha_{\epsilon}(P(\cdot;\nu))$ and  $\alpha^{\stan}_{\epsilon}({\mathcal A}(\nu))$,
consider the $20\times 20$ damping optimization problem in Example \ref{lar:ex1}
with $\epsilon = 0.05$ and 
the interval $\underline{\Omega} = [0,100]$. In the definition of $\alpha_{\epsilon}(P(\cdot;\nu))$,
the weights are chosen as $(w_m, w_c, w_k) = (0, 1, 1)$. Moreover, as $\| M^{-1} \|_2 = 1$,
we have $\underline{\epsilon} = \epsilon\| M^{-1} \|_2  = \epsilon$, so
it follows from Theorem \ref{thm:psa_poly_vs_stan} for that all $\nu\in \underline{\Omega}$ we have
$\alpha_{\epsilon}(P(\cdot;\nu)) \leq \alpha^{\stan}_{\epsilon}({\mathcal A}(\nu))$.

Minimizing $\alpha_{\epsilon}(P(\cdot;\nu))$ over $\nu \in \underline{\Omega}$ using Algorithm~\ref{alg:sub_fw_min} yields
$\nu_\ast = 66.42$ as its global minimizer, 
and $\alpha_{\epsilon}(P(\cdot;\nu_\ast)) = 0.0012$
as the globally smallest value of $\alpha_{\epsilon}(P(\cdot;\nu))$.
On the other hand, the minimization of $\alpha^{\stan}_{\epsilon}({\mathcal A}(\nu))$ over $\nu \in \underline{\Omega}$
using the framework in \cite{AliM22}, in particular the software in \cite{AliM22b}, leads to 
$\nu^{\stan}_\ast = 55.47$ as the global minimizer, 
and $\alpha^{\stan}_{\epsilon}({\mathcal A}(\nu^{\stan}_\ast)) = 0.1155$
as the globally smallest value of $\alpha^{\stan}_{\epsilon}({\mathcal A}(\nu))$.

The accuracy of these computed results is illustrated by the plots of 
$\alpha_{\epsilon}(P(\cdot;\nu))$ and  $\alpha^{\stan}_{\epsilon}({\mathcal A}(\nu))$
in Figure \ref{fig:compare_psa}. The computed global minimizers $\nu_\ast = 66.42$ and 
$\nu^{\stan}_\ast = 55.47$
are away from each other. Moreover, $\alpha^{\stan}_{\epsilon}({\mathcal A}(\nu))$
is considerably larger than $\alpha_{\epsilon}(P(\cdot;\nu))$ by a magnitude of at least ten globally
for all $\nu \in \underline{\Omega}$. A comparison of the spectral abscissa $\alpha(\nu_\ast) = -0.00844$
and $\alpha(\nu^{\stan}_\ast) = -0.00843$ at the computed global minimizers indicates that
the system $M \ddot{x}(t) + C(\nu) \dot{x}(t) + K x(t) = 0$ has slightly better asymptotic behavior 
at $\nu = \nu_\ast$ compared to $\nu = \nu^{\stan}_\ast$.

\begin{figure}
 \centering
	\begin{tabular}{cc}
		\hskip -1.2ex
			\includegraphics[width = .49\textwidth]{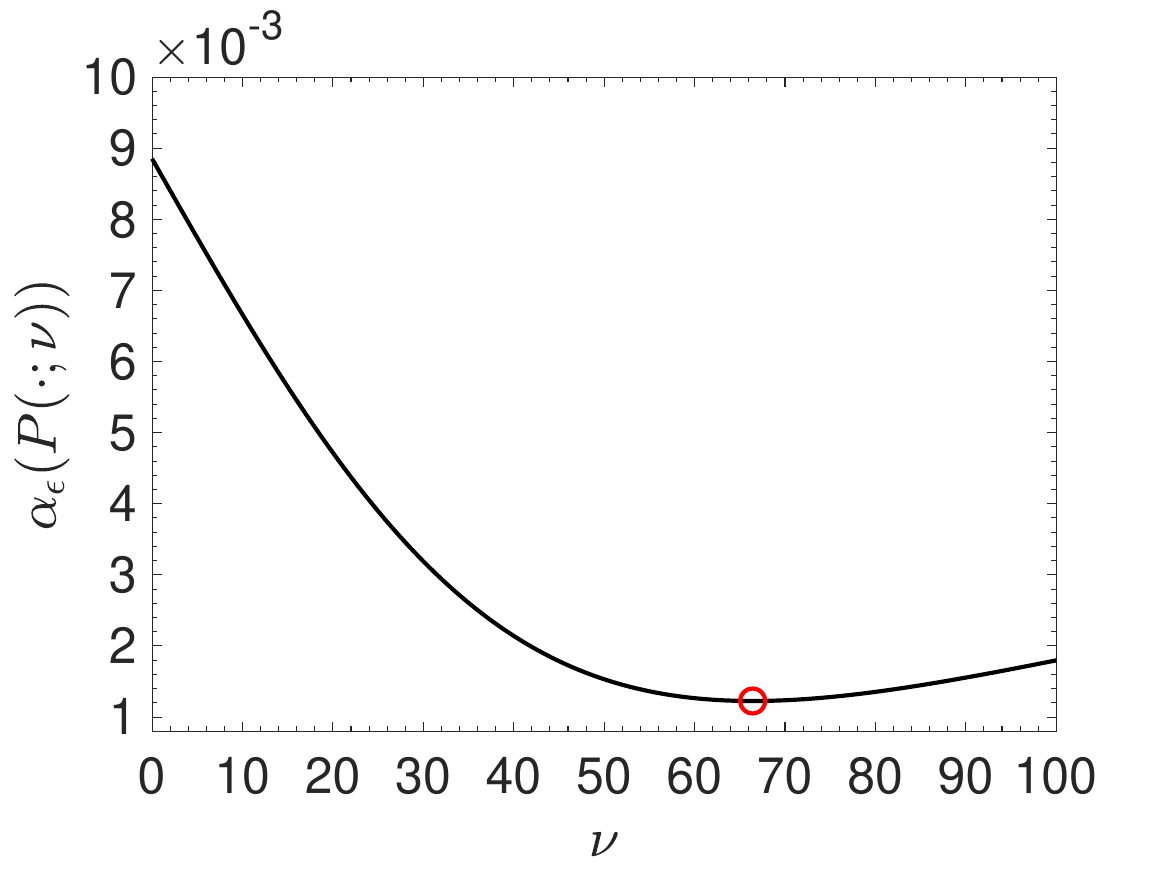} & 
			\includegraphics[width = .49\textwidth]{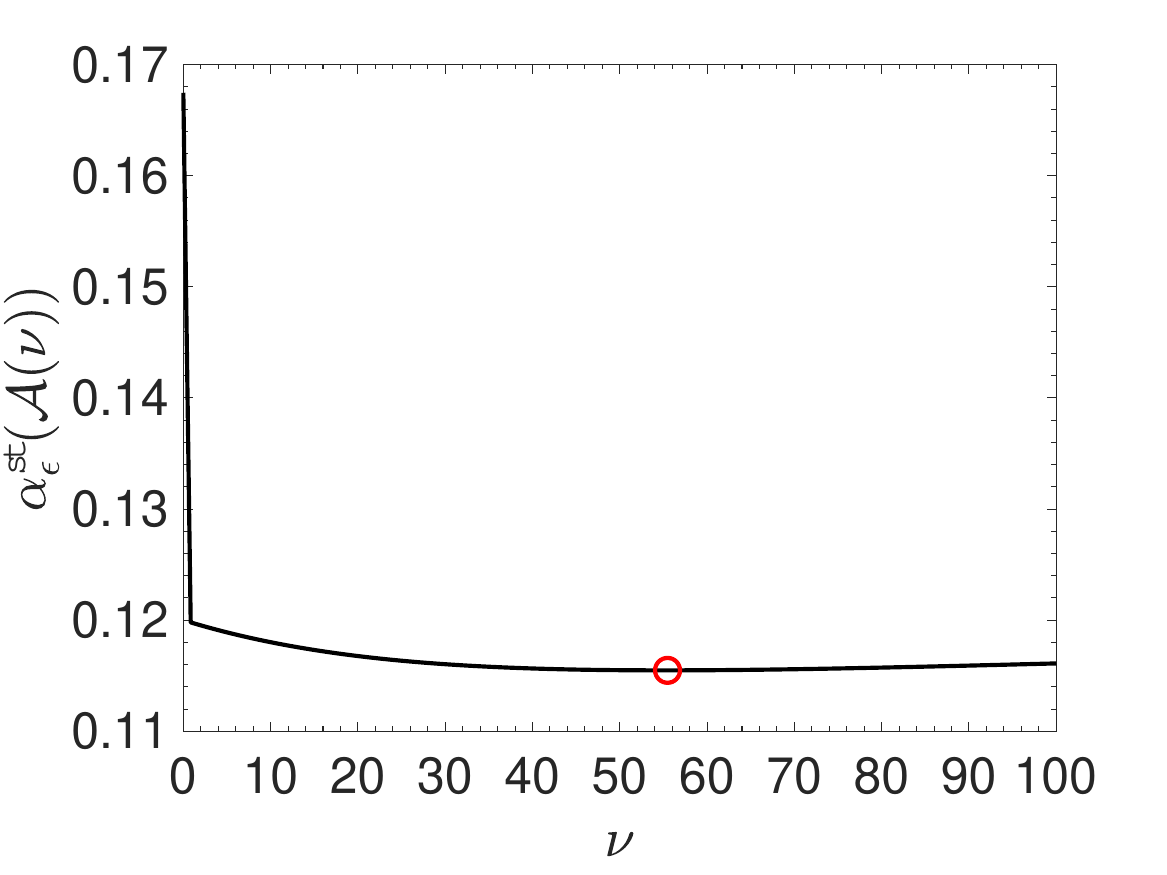} 
	\end{tabular}
		\caption{  Comparison of $\alpha_{\epsilon}(P(\cdot ; \nu))$
		with $\alpha_{\epsilon}^{\stan}({\mathcal A}(\nu))$ on the 
		damping optimization problem in Example \ref{lar:ex1}. In both plots, the red circle marks the
		computed global minimizer. Left Plot: $\alpha_{\epsilon}(P(\cdot ; \nu))$ as a function
		of $\nu$ for $(w_m , w_c , w_k) = (0, 1, 1)$.
		Right Plot: $\alpha_{\epsilon}^{\stan}({\mathcal A}(\nu))$ as a function of $\nu$. 
		}
		\label{fig:compare_psa}
\end{figure}
}
\end{example}

\section{Concluding Remarks}\label{sec:conclusion}
We have proposed numerical approaches for the minimization of the
$\epsilon$-pseudospecrtral abscissa $\alpha_{\epsilon}(P(\, \cdot \, ; \nu))$
over the parameters $\nu$ that belong to a compact set $\underline{\Omega} \in {\mathbb R}^{\sd}$
for a given parameter-dependent quadratic matrix polynomial $P(\, \cdot \, ; \nu)$ 
and for a prescribed $\epsilon > 0$. When $P(\, \cdot \, ; \nu)$ is of small size,
the minimization of $\alpha_{\epsilon}(P(\, \cdot \, ; \nu))$ is performed by means of a
global optimization technique that can cope with the nonsmoothness of 
$\alpha_{\epsilon}(P(\, \cdot \, ; \nu))$ (or alternatively with BFGS if there are
many optimization parameters) coupled with a globally convergent criss-cross
algorithm to compute the objective $\alpha_{\epsilon}(P(\, \cdot \, ; \nu))$.
Otherwise, when $P(\, \cdot \, ; \nu)$ is not small, it is minimized by applying
a subspace framework that operates on a restriction 
of $P(\, \cdot \, ; \nu)$ to a small subspace ${\mathcal V}$.
We have provided formal arguments that indicate the convergence of this framework
to a global minimizer of $\alpha_{\epsilon}(P(\, \cdot \, ; \nu))$. Numerical
experiments on damping optimization examples
show that the proposed methods can cope with 
the minimization of $\alpha_{\epsilon}(P(\cdot ; \nu))$
accurately and efficiently for problems of size at least a few hundreds in practice.

Efficient and reliable $\epsilon$-pseudospectral abscissa computations in the large-scale
setting are one of the main obstacles standing in the way to deal with even larger $\epsilon$-pseudospectral
abscissa minimization problems for quadratic eigenvalue problems, e.g. those related to the
robust stability of finite element discretizations of a disk brake \cite{GraMQ2015}.
The framework we rely on (i.e. Algorithm~\ref{alg:sub_fw}) for these larger $\epsilon$-pseudospectral
abscissa computations is a remedy, but still needs improvements to avoid local convergence
such as a clever initialization strategy, possibly making use of the perturbation theory for quadratic 
eigenvalue problems. 

\smallskip

\noindent
\textbf{Acknowledgements.} The authors are grateful to two anonymous referees
for their comments and suggestions, which helped the authors to improve the manuscript.



\bibliography{poly_pspa_refs}


\end{document}